\documentclass{amsart}

\usepackage{amsfonts, amssymb, amsmath, amsthm}                               
\usepackage{url}                                                              
\usepackage{color}
\usepackage{slashed}
\usepackage{comment}
\usepackage{enumerate}
\numberwithin{equation}{section}
\newtheorem{theorem}{Theorem}[section]      	      	                        
\newtheorem{definition}[theorem]{Definition}   	      	                      
\newtheorem{lemma}[theorem]{Lemma}     	       	      	      	      	      
\newtheorem{proposition}[theorem]{Proposition} 	      	      	      	      

\theoremstyle{remark}
\newtheorem{remark}[theorem]{Remark}                                                  

\numberwithin{equation}{section}                                              
\numberwithin{theorem}{section}                                               
\numberwithin{figure}{section}                                                

\newcommand{\mc}[1]{\mathcal{#1}}                                             

\newcommand{\R}{\mathbb{R}}                                                   
\newcommand{\Sph}{\mathbb{S}}                                                 

\newcommand{\snabla}{\slashed{\nabla}}
\newcommand{\pd}{\partial}
\newcommand{\nab}{\nabla}                                                        

\newcommand{\al}{\alpha}
\newcommand{\be}{\beta}
\newcommand{\de}{\delta}
\newcommand{\ep}{\varepsilon}

\newcommand{\ka}{\kappa}
\newcommand{\la}{\lambda}

\newcommand{\De}{\Delta}
\newcommand{\Ga}{\Gamma}

\renewcommand\leq\leqslant
\renewcommand\geq\geqslant

\allowdisplaybreaks

\begin{document}

\title[Carleman estimates for waves with critically singular
potentials]{Carleman estimates with sharp weights\\
  and boundary observability for wave operators\\ with critically singular potentials}

\author{Alberto Enciso}
\address{Instituto de Ciencias Matem\'aticas, Consejo Superior de Investigaciones Cient\'\i ficas, 28049 Madrid, Spain}
\email{aenciso@icmat.es}

\author{Arick Shao}
\address{School of Mathematical Sciences,
Queen Mary University, London, United Kingdom}
\email{a.shao@qmul.ac.uk}

\author{Bruno Vergara}
\address{Instituto de Ciencias Matem\'aticas, Consejo Superior de Investigaciones Cient\'\i ficas, 28049 Madrid, Spain}
\email{bruno.vergara@icmat.es}

\begin{abstract}
  We establish a new family of Carleman inequalities for wave
  operators on cylindrical spacetime domains containing a potential
  that is critically singular, diverging as an inverse square on all
  the boundary of the domain.  These estimates are sharp in the sense
  that they capture both the natural boundary conditions and the
  natural $H^1$-energy.  The proof is based around three key
  ingredients: the choice of a novel Carleman weight with rather
  singular derivatives on the boundary, a generalization of the
  classical Morawetz inequality that allows for inverse-square
  singularities, and the systematic use of derivative operations
  adapted to the potential. As an application of these estimates, we
  prove a boundary observability property for the associated wave
  equations.
\end{abstract}

\maketitle

\section{Introduction}

Our objective in this paper is to derive Carleman estimates for wave operators with critically singular potentials, that is, with potentials that scale like the principal part of the operator.
More specifically, we are interested in the case of potentials that diverge as an inverse square on a convex hypersurface.

For the present paper, we consider the model operator
\begin{equation}\label{operator}
\Box_\ka := \square + \frac{ \ka ( 1 - \ka ) }{(1-|x|)^2} \text{,}
\end{equation}
where $\square:=-\pd_{tt}+\De$ is the wave operator, the spatial domain is the unit ball $B_1$ of $\R^n$, and the constant parameter $\ka \in \R$ measures the strength of the potential.

\subsection{Background}

To understand why we say ``sharp'', let us consider the Cauchy problem associated with this operator,
\begin{align}\label{Cauchyprob}
  \begin{split}
\Box_\kappa u = 0 \quad &\text{in } ( -T, T ) \times B_1 \text{,} \\
  u(0,x)=u_0(x) \text{,} \qquad &\pd_tu(0,x)=u_1(x) \text{.}
  \end{split}
\end{align}
In spherical coordinates, the equation reads as
\[
- \pd_{tt} u + \pd_{rr} u + \frac{n-1}{r} \pd_r u + \frac{ \ka ( 1 - \ka ) }{ (1-r)^2 } u + \frac{1}{r^2} \De_{ \Sph^{n-1} } u=0 \text{,}
\]
where $\De_{\Sph^{n-1}}$ denotes the Laplacian on the unit sphere. The potential is critically singular at $r=1$, where, according to the classical theory of Frobenius for ODEs, the characteristic exponents of this equation are $\ka$ and $1-\ka$.
Therefore, if $\ka$ is not a half-integer (which ensures that logarithmic branches will not appear), solutions to the equation are expected to behave either like $(1-r)^{\ka}$ or $(1-r)^{1-\ka}$ as $r\nearrow1$.

As one can infer by plugging these powers in the energy associated with \eqref{Cauchyprob},
\begin{equation}\label{energy1}
\int_{ \{ t \} \times B_1 } \left\{ ( \pd_t u )^2+ (1-r)^{2\ka} \left| \nabla_{ x } [ (1-r)^{-\ka} u ] \right|^2 \right\} \text{,}
\end{equation}
the equation admits exactly one finite-energy solution when $\ka\leq-\frac12$, no finite-energy
solutions when $\ka\geq \frac12$, and infinitely many
finite-energy solutions when
\begin{equation}\label{interval}
-\frac12<\ka<\frac12\,.
\end{equation}

In this range \eqref{interval} of the parameter, which we consider in this paper, one must impose a (Dirichlet, Neumann, or Robin) boundary condition on $( -T, T ) \times \pd B_1$.
This is constructed in terms of the natural Dirichlet and Neumann traces, which now include weights and are defined as the limits
\begin{equation}\label{BCs}
\mc{D}_\ka u := (1-r)^{-\ka} u|_{r=1} \text{,} \qquad \mc{N}_\ka u := -(1-r)^{2\ka} \pd_r [ (1-r)^{-\ka} u ] |_{r=1} \text{.}
\end{equation}

Notice that singular weights depending on~$\ka$ appear everywhere in
this problem, and that all the associated quantities reduce to the standard ones in the absence of the
singular potential, i.e., when $\ka=0$.
A more detailed discussion of the boundary asymptotics of solutions to \eqref{Cauchyprob} is given in the next section.

The Carleman estimates that we will derive in this paper are sharp, in that the weights that appear capture both the optimal decay rate of the
solutions near the boundary, as well as the natural energy~\eqref{energy1} that appears in the well-posedness theory for the equation.
As we will see, this property is not only desirable but also
essential for applications such as boundary observability.

\subsection{Some Existing Results}

The dispersive properties of wave equations with potentials that diverge as an inverse square at one point~\cite{Vega,Burq2} or an a (timelike) hypersurface~\cite{Bachelot} have been thoroughly studied, as critically singular potentials are notoriously difficult to analyze.
Moreover, a well-posedness theory for a diverse family of boundary conditions was developed for the range \eqref{interval} in \cite{Warnick}.

In the case of one spatial dimension, the observability and controllability of wave equations with critically singular potentials have also received considerable attention, in the guise of the degenerate wave equation
$$
\pd_{tt}v-\pd_z (z^\al \pd_zv)=0\,,
$$
where the variable $z$ takes values in the positive half-line and the
parameter~$\al$ ranges over the interval~$(0,1)$; see~\cite{Gueye} and the references therein.
Indeed, it is not difficult to show that one can relate equations in this form with the operator $\Box_\ka$ in one dimension through a suitable change of variables, with the parameter $\ka$ being now some function of the power $\al$.
The methods employed in those references, which rely on the spectral analysis of a one-dimensional Bessel-type operator, provide very precise observability and controllability results.

Another fruitful strategy for obtaining observability inequalities for a wide variety of PDEs is via Carleman-type estimates; see \cite{Tataru0, Tataru_Thesis} for some earliest applications, as well as \cite{LTX, Zhang} for wave equations.
On the other hand, no related Carleman estimates that are applicable to observability results for $\Box_\kappa$ have been found.
This manifests itself in two important limitations: firstly, the available inequalities are not robust under
perturbations on the coefficients of the equation, and secondly, the method of proof cannot be extended to higher-dimensional situations.

Recent results for different notions of observability for parabolic equations with inverse square potentials, which are based on Carleman and multiplier methods, can be found, e.g., in \cite{BZuazua, VZuazua}.
Related questions for wave equations with singularities all over the boundary have been presented as very challenging in the open problems section of \cite{BZuazua}.
As stressed there, the boundary singularity makes the multiplier approach extremely tricky.

In general, one would not expect Carleman estimates to behave well with singular potentials such as $\ka (1 - \ka) ( 1 - r )^{-2}$.
Since the singularity in the potential scales just as $\Box$, there is no hope in absorbing it into the estimates by means of a perturbative argument.
Indeed, Carleman estimates generally assume~\cite{DosSantos, KT3, Tataru} that the potential is at least in $L^{ (n + 1) / 2 }$, but this condition is not satisfied here.

Consequently, we must view this singular potential as a principal term and instead derive a Carleman estimate for the modified wave operator $\Box_\kappa$ in \eqref{operator}.
Such estimates for other modified wave operators involving lower-derivative terms have been obtained, for instance, in \cite{BBE, LLZ}.
However, a key difference in the present situation is the specially weighted forms \eqref{BCs} of our natural boundary traces.
In particular, to capture the Neumann trace, our Carleman estimates must also involve weights that become singular at the boundary $( -T, T ) \times \partial B_1$.

Carleman estimates with degenerating weights have been applied extensively in the context of strong unique continuation problems for PDEs.
Examples in the literature include \cite{Arons, Carleman, KT, Sogge} for elliptic equations and \cite{EF, KT2, LO} for parabolic equations; see also \cite{AS, ASS} for analogous problems for hyperbolic equations.
On the other hand, the weights used here will be very different in nature to those from strong unique continuation results, since we will require degeneracies at a very specific power in order to pick out the Neumann traces described in \eqref{BCs}.

Finally, let us mention that a setting which is closely related to ours is that of linear wave equations on asymptotically anti-de Sitter spacetimes, which are conformally equivalent to analogues of \eqref{operator} on curved backgrounds.
It is worth mentioning that waves on anti-de Sitter spaces have attracted
considerable attention in the recent years due to their connection to cosmology, see e.g.~\cite{Bachelot, Vergara, Enciso,HS, Warnick} and the
references therein. 

Carleman estimates for linear waves were established in this asymptotically anti-de Sitter setting in \cite{HS, HS2}, for the purposes of studying their unique continuation properties from the conformal boundary.
In particular, these estimates capture the natural Dirichlet and Neumann data (i.e., the analogues of \eqref{BCs}).
On the other hand, the Carleman estimates in \cite{HS, HS2} are local in nature and apply only to a neighborhood of the conformal boundary, and they do not capture the naturally associated $H^1$-energy.
As a result, these estimates would not translate into corresponding observability results.

\subsection{The Carleman Estimates}

The main result of the present paper is a novel family of Carleman inequalities for the operator~\eqref{operator} that capture both the natural boundary weights and the natural $H^1$-energy described above.
To the best of our knowledge, these are the first available Carleman estimates for an operator with such a strongly singular potential that also captures the natural boundary data and energy.
Moreover, our estimates hold in all spatial dimensions, except for $n = 2$.

A simplified version of our main estimates can be stated as follows:

\begin{theorem}\label{T.Carleman0}
Let $B_1$ denote the unit ball in $\R^n$, with $n \neq 2$, and fix $- \frac{1}{2} < \kappa < 0$.
Moreover, let $u: ( -T, T ) \times B_1 \rightarrow \R$ be a smooth function, and assume:
\begin{enumerate}[i)]
\item The Dirichlet trace $\mc{D}_\kappa u$ of $u$ vanishes.

\item $u$ ``has the boundary asymptotics of a sufficiently regular, finite energy solution of \eqref{Cauchyprob}".
In particular, the Neumann trace $\mc{N}_\kappa u$ of $u$ exists and is finite.

\item There exists $\delta > 0$ such that $u (t) = 0$ for all $T - \delta \leq |t| < T$.
\end{enumerate}
Then, for $\la \gg 1$ large enough, independently of $u$, the following inequality holds:
\begin{align} \label{Carleman0}
&\la \int_{ ( -T, T ) \times \pd B_1 } e^{ 2 \lambda f } ( \mc{N}_\ka u)^2 + \int_{ ( -T, T ) \times B_1}
                                  e^{2\la f}  (\Box_\kappa u )^2 \\
\notag &\quad  \gtrsim \la \int_{ ( -T,T ) \times B_1 } e^{ 2 \la f}
         \Big[ (\pd_t u)^2+(1-|x|)^{2\ka}\,\big| \nabla_{ x } [ (1-|x|)^{-\ka} u ] \big|^2 \Big]\\
\notag &\quad\qquad + | \kappa | \la^3 \int_{ ( -T,T ) \times B_1 } e^{ 2 \lambda f } ( 1 - |x| )^{6\ka-1} u^2 \text{,}
\end{align}
where $f$ is the weight
\begin{equation} \label{weight}
f(t, x) := -\frac{1}{1+2\ka}(1-|x|)^{1+2\ka}-ct^2 \text{,}
\end{equation}
with a suitably chosen positive constant $c$.
\end{theorem}

A more precise, and slightly stronger, statement of our main Carleman estimates is given further below in Theorem \ref{T.Carleman}.


\begin{remark}
Note that in Theorem \ref{T.Carleman0}, we restricted our strength parameter $\kappa$ to the range $- \frac{1}{2} < \kappa < 0$.
This was imposed for several reasons:
\begin{enumerate}[i)]
\item First, a restriction to the values \eqref{interval} was needed, as this is the range for which a robust well-posedness theory exists \cite{Warnick} for the equation \eqref{Cauchyprob}.

\item The case $\kappa = 0$ is simply the standard free wave equation, for which the existence of Carleman and observability estimates is well-known.

\item On the other hand, the aforementioned spectral results \cite{Gueye} in the $(1 + 1)$-dimensional setting suggest that the analogue of \eqref{Carleman0} is false when $\kappa > 0$.
\end{enumerate}
\end{remark}

\begin{remark}
The constant $c$ in \eqref{weight} is closely connected to the total timespan needed for an observability estimate to hold; see Theorem \ref{T.Observability0} below.
In Theorem~\ref{T.Carleman0}, this $c$ depends on $n$, as well as on $\kappa$ when $n = 3$.
\end{remark}

\begin{remark}
The precise formulation of $u$ in Theorem \ref{T.Carleman0} having the ``expected boundary asymptotics of a solution of \eqref{Cauchyprob}" is given in Definition \ref{admissible} and is briefly justified in the discussion following Definition \ref{admissible}.
\end{remark}

\begin{remark}
One can further strengthen \eqref{Carleman0} to include additional positive terms on the right-hand side that depend on $n$; see Theorem \ref{T.Carleman}.
\end{remark}

\subsection{Ideas of the Proof}

We now discuss the main ideas behind the proof of Theorem \ref{T.Carleman0} (as well as the more precise Theorem \ref{T.Carleman}).
In particular, the proof is primarily based around three ingredients.

The first ingredient is to adopt derivative operations that are well-adapted to our operator $\Box_\kappa$.
In particular, we make use of the ``twisted" derivatives that were pioneered in \cite{Warnick}.
The main observation here is that $\Box_\kappa$ can be written as
\[
\Box_\kappa = - \bar{D} D + \text{l.o.t.} \text{,}
\]
where $D$ is the conjugated (spacetime) derivative operator,
\[
D = D_{ t, x } = ( 1 - | x | )^\kappa \nabla_{ t, x } ( 1 - | x | )^{ - \kappa } \text{,}
\]
where $-\bar{D}$ is the ($L^2$-)adjoint of $D$, and where ``l.o.t." represents lower-order terms that can be controlled by more standard means.

As a result, we can view $D$ as the natural derivative operation for $\Box_\kappa$.
For instance, the twisted $H^1$-energy \eqref{energy1} associated with the Cauchy problem \eqref{Cauchyprob} is best expressed purely in terms
of $D$ (in fact, this energy is conserved for the equation $\bar{D} D u = 0$).
Similarly, in our Carleman estimates \eqref{Carleman0} and their proofs, we will always work with $D$-derivatives, rather than the usual derivatives, of $u$.
This helps us to better exploit the structure of $\Box_\kappa$.

The second main ingredient in the proof of Theorem \ref{T.Carleman0} is the classical Morawetz multiplier estimate for the wave equation.
This estimate was originally developed in \cite{Morawetz} in order to establish integral decay properties for waves in $3$ spatial dimensions.
Analogous estimates hold in higher dimensions as well; see \cite{Tao}, as well as \cite{Keith} and references therein for more recent extensions of Morawetz estimates.

At the heart of the proof of Theorem \ref{T.Carleman0} lies a generalization of the classical Morawetz estimate from $\Box$ to $\Box_\kappa$.
In keeping with the preceding ingredient, we derive this inequality by using the aforementioned twisted derivatives in the place of the usual derivatives.
This produces a number of additional singular terms, which we must arrange so that they have the required positivity.

Finally, our generalized Morawetz bound is encapsulated within a larger Carleman estimate, which is proved using geometric multiplier arguments (see, e.g., \cite{AS, HS, HS2, IK, LTX}).
Again, we adopt twisted derivatives throughout this process, and we must obtain positivity for many additional singular terms that now appear.

\begin{remark}
That Theorem \ref{T.Carleman0} fails to hold for $n = 2$ can be traced to the fact that the classical Morawetz breaks down for $n = 2$.
In this case, the usual multiplier computations yield a boundary term at $r = 0$ that is divergent.
\end{remark}

\begin{remark}
Both the Carleman estimates \eqref{Carleman0} and the underlying Morawetz estimates crucially depend on the domain being spherically symmetric.
As a result, Theorem \ref{T.Carleman0} only holds when the spatial domain is an open ball.
We defer questions of whether Theorem \ref{T.Carleman0} extends more general domains to future papers.
\end{remark}

\subsection{The Carleman Weight}

For our estimate \eqref{Carleman0}, we make use of a novel Carleman weight \eqref{weight} that is especially adapted to the operator $\Box_\kappa$.

Recall that in the standard Carleman-based proofs of observability for wave equations, one employs Carleman weights of the form
\[
f_\ast ( t, x ) = | x - x_0 |^2 - c t^2 \text{,} \qquad 0 < c < 1 \text{.}
\]
Here, the term $| x - x_0 |^2$ can be roughly interpreted as the estimate being centered about the point $x_0$.
In contrast, in \eqref{weight}, the spatial term of $f$ is replaced by a power of $1 - | x |$.
This can be viewed as our estimate being centered about the whole boundary $\partial B_1$, where $\Box_\kappa$ becomes singular.

The next point of interest is the exponent $1 + 2 \kappa$ in \eqref{weight}.
Such a power, which leads to rather singular terms at $r = 1$, seems necessary in our estimates in order to extract the Neumann boundary data, which contains a specific power of $1 - | x |$.

We also remark that the weight $f$ in \eqref{weight} is strongly pseudo-convex (as defined in \cite[Definition 28.3.1]{Hormander}) with respect to the standard wave operator $\Box$.
As is well-known, this is necessary in order for such a Carleman-type estimate to hold.
In our current context, the pseudo-convexity is captured by the quantity $\nabla^2 f + z \cdot g$ from our multiplier identity \eqref{gralmult}, which can be shown to be positive-definite in the directions tangent to the level sets of $f$; see also Remark \ref{rmk.pseudoconvex}.

In fact, the most difficult obstructions to our Carleman estimate arise not from pseudo-convexity.
(One can see that $f$ becomes infinitely pseudo-convex at the boundary $( -T, T ) \times \partial B_1$.)
Rather, the main difficulty comes from ensuring that the key singular bulk terms arising from the generalized multiplier estimates all possess good sign.
For this, we need more than the pseudo-convexity of the Carleman weight; this is the reason we restrict our analysis to the spatial domain $B_1$.

\subsection{Observability}

The breadth of applications of Carleman estimates to a wide range of PDEs \cite{DZZ, Tataru1} is remarkable.
Examples include unique continuation, control theory, inverse problems, as well as showing the absence of embedded eigenvalues in the continuous spectrum of Schr\"{o}dinger operators.

In this paper, we demonstrate one particular consequence of Theorem~\ref{T.Carleman0}: the boundary observability of linear waves involving a critically singular potential.
Roughly speaking, a boundary observability estimate shows that the energy of a wave confined in a bounded region can be estimated quantitatively only by measuring its boundary data over a large enough time interval.

The key point is again that our Carleman estimates \eqref{Carleman0} capture the natural boundary data and energy associated with our singular wave operator.
As a result of this, Theorem \ref{T.Carleman0} can be combined with
standard arguments in order to prove the following rough
statement:~solutions to the wave equation with a critically singular
potential on the boundary of a cylindrical domain satisfy boundary observability estimates, provided that the observation is made over a large enough timespan. 

A rigorous statement of this observability property is given in the subsequent theorem.
Notice that, due to energy estimates that we will show later, it is
enough to control the twisted $H^1$-norm of the solution at time zero:

\begin{theorem}\label{T.Observability0}
Let $B_1$, $n$, and $\kappa$ be as in Theorem \ref{T.Carleman0}.
Moreover, let $u$ be a smooth and real-valued solution of the wave equation
\begin{equation}
\label{Obs_wave_0} \Box_\ka u = X \cdot D u + V u
\end{equation}
on the cylinder $( -T, T ) \times B_1$, where $X$ is a bounded (spacetime) vector field, and where $V$ is a bounded scalar potential.
Furthermore, suppose $u$ satisfies:
\begin{enumerate}[i)]
\item $\mc{D}_\kappa u = 0$.

\item $u$ ``has the boundary asymptotics of a sufficiently regular, finite energy solution of \eqref{Obs_wave_0}".
In particular, the Neumann trace $\mc{N}_\kappa u$ of $u$ exists and is finite.
\end{enumerate}
Then, for sufficiently large $T$, the following observability estimate holds for $u$:
\begin{equation}\label{Obs0}
\int_{ ( -T,T ) \times \pd B_1 } ( \mc{N}_\ka u )^2 \gtrsim \int_{ \{ 0 \} \times B_1 }\Big[ ( \pd_t u )^{2} + | (1 - |x| )^\kappa \nabla_x [ (1-|x|)^{-\ka} u ] |^2 + u^2 \Big] \,.
\end{equation}
\end{theorem}

Again, a more precise (and slightly more general) statement of the observability property can be found in Theorem \ref{T.Observability}.

\begin{remark}
The required timespan $2 T$ in Theorem \ref{T.Observability0} can be shown to depend on $n$, as well as on $\kappa$ when $n = 3$.
This is in direct parallel to the dependence of $c$ in Theorem \ref{T.Carleman0}.
See Theorem \ref{T.Observability} for more precise statements.
\end{remark}

\begin{remark}
Once again, a precise statement of the expected boundary asymptotics for $u$ in Theorem \ref{T.Observability0} is given in Definition \ref{admissible}.
\end{remark}

\begin{remark}
If $\Box_\kappa$ in Theorem \ref{T.Observability0} is replaced by $\Box$ (that is, we consider non-singular wave equations), then observability holds for any $T > 1$.
This can be deduced from either the geometric control condition of \cite{BLR} (see also \cite{BG, Macia}) or from standard Carleman estimates \cite{BBE, LTX, Zhang}.
To our knowledge, the optimal timespan for the observability result in Theorem \ref{T.Observability0} is not known.
\end{remark}

\begin{remark}
For non-singular wave equations, standard observability results also involve observation regions that contain only part of the boundary \cite{BLR, BG, LTX, Lions1}.
On the other hand, as our Carleman estimates \eqref{Carleman0} are centered about the origin, they only yield observability results from the entire boundary.
Whether partial boundary observability results also hold for the singular wave equation in Theorem~\ref{T.Observability0} is a topic of further investigation.
\end{remark}

\subsection{Outline of the Paper}

In Section~\ref{S.asymptotics}, we list some definitions that will be pertinent to our setting, and we establish some general properties that will be useful later on.
Section~\ref{S.multipliers} is devoted to the multiplier inequalities that are fundamental to our main Theorem~\ref{T.Carleman0}.
In particular, these generalize the classical Morawetz estimates to wave equations with critically singular potentials.
In Section~\ref{S.Carleman}, we give a precise statement and a proof of our main Carleman estimates (see Theorem \ref{T.Carleman}).
Finally, our main boundary observability result (see Theorem \ref{T.Observability}) is stated and proved in Section \ref{S.Observability}.

\section{Preliminaries} \label{S.asymptotics}

In this section, we record some basic definitions, and we establish the notations that we will use in the rest of the paper.
In particular, we define weights that capture the boundary behavior of solutions to wave equations rendered by~$\Box_\ka$.
We also define twisted derivatives constructed using the above weights, and we recall their basic properties.
Furthermore, we prove pointwise inequalities in terms of these twisted derivatives that will later lead to Hardy-type estimates.

\subsection{The Geometric Setting}

Our background setting is the spacetime $\R^{ 1 + n }$.
As usual, we let $t$ and $x$ denote the projections to the first and the last $n$ components of $\R^{ 1 + n }$, respectively, and we let $r := | x |$ denote the radial coordinate.

In addition, we let $g$ denote the Minkowski metric on $\R^{ 1 + n }$.
Recall that with respect to polar coordinates, we have that
\[
g = - dt^2 + dr^2 + r^2 g_{ \Sph^{n-1} } \text{,}
\]
where  $g_{ \Sph^{n-1} }$ denotes the metric of the $(n-1)$-dimensional unit sphere.
Henceforth, we use the symbol $\nabla$ to denote the $g$-covariant derivative, while we use $\snabla$ to represent the induced angular covariant derivative on level spheres of $( t, r )$.
As before, the wave operator (with respect to $g$) is defined as
\[
\Box = g^{ \alpha \beta } \nabla_{ \alpha \beta } \text{.}
\] 

As it is customary, we use lowercase Greek letters for spacetime indices over $\R^{ n + 1 }$ (ranging from $0$ to $n$), lowercase Latin letters for spatial indices over $\R^n$ (ranging from $1$ to $n$), and uppercase Latin letters for angular indices over $\Sph^{ n - 1 }$ (ranging from $1$ to $n - 1$).
We always raise and lower indices using $g$, and we use the Einstein summation convention for repeated indices.

As in the previous section, we use $B_1$ to denote the open unit ball in $\R^n$, representing the spatial domain for our wave equations.
We also set
\begin{equation}
\label{domain} \mc{C} := ( -T, T ) \times B_1 \text{,} \qquad T > 0 \text{,}
\end{equation}
corresponding to the cylindrical spacetime domain.
In addition, we let
\begin{equation}
\label{domain_tbdry} \Gamma := ( -T, T ) \times \partial B_1
\end{equation}
denote the timelike boundary of $\mc{C}$.

To capture singular boundary behavior, we will make use of weights depending on the radial distance from $\partial B_1$.
Toward this end, we define the function
\begin{equation}
\label{y} y: \R^{ 1 + n } \rightarrow \R \text{,} \qquad y := 1 - r \text{.}
\end{equation}
From direct computations, we obtain the following identities for $y$:
\begin{align}
\label{y_id} \nabla^\alpha y \nabla_\alpha y = 1 \text{,} &\qquad \nabla^{ \alpha \beta } y \nabla_\alpha y \nabla_\beta y = 0 \text{,} \\
\notag \Box y = - ( n - 1 ) r^{-1} \text{,} &\qquad \nabla^\alpha y \nabla_\alpha ( \Box y ) = - ( n - 1 ) r^{-2} \text{,} \\
\notag \Box^2 y = ( n - 1 ) ( n - 3 ) r^{-3} \text{,} &\qquad \nabla^{ \alpha \beta } y \nabla_{ \alpha \beta } y = ( n - 1 ) r^{-2} \text{.}
\end{align}

\subsection{Twisted Derivatives}

From here on, let us fix a constant
\begin{equation}
\label{kappa} - \frac{1}{2} < \kappa < 0 \text{,}
\end{equation}
and let us define the twisted derivative operators
\begin{align}
\label{twisted} D \Phi &:= y^\kappa \nabla ( y^{ - \kappa } \Phi ) = \nabla \Phi - \frac{ \kappa }{ y } \nabla y \cdot \Phi \text{,} \\
\notag \bar{D} \Phi &:= y^{ - \kappa } \nabla ( y^\kappa \Phi ) = \nabla \Phi + \frac{ \kappa }{ y } \nabla y \cdot \Phi \text{,}
\end{align}
where $\Phi$ is any spacetime tensor field.
Observe that $- \bar{D}$ is the formal ($L^2$-)adjoint of $D$.
Moreover, the following (tensorial) product rules hold for $D$ and $\bar{D}$:
\begin{equation}
\label{prod_rule} D ( \Phi \otimes \Psi ) = \nabla \Phi \otimes \Psi + \Phi \otimes D \Psi \text{,} \qquad \bar{D} ( \Phi \otimes \Psi ) = \nabla \Phi \otimes \Psi + \Phi \otimes \bar{D} \Psi \text{.}
\end{equation}

In addition, let $\Box_y$ denote the $y$-twisted wave operator:
\begin{equation}
\label{Box_y} \Box_y := g^{ \alpha \beta } \bar{D}_\alpha D_\beta \text{.}
\end{equation}
A direct computation shows that $\Box_y$ differs from the singular wave operator $\Box_\kappa$ from \eqref{operator} by only a lower-order term.
More specifically, by \eqref{y_id} and \eqref{twisted},
\begin{align}
\label{Box_y_kappa} \Box_y &= \Box + \frac{ \kappa ( 1 - \kappa ) \cdot \nabla^\alpha y \nabla_\alpha y }{ y^2 } - \frac{ \kappa \cdot \Box y }{ y } \\
\notag &= \Box_\ka + \frac{ ( n - 1 )\kappa }{ r y } \text{.}
\end{align}

In particular, \eqref{Box_y_kappa} shows that, up to a lower-order correction term, $\Box_y$ and $\Box_\ka$ can be used interchangeably.
In practice, the derivation of our estimates will be carried out in terms of $\Box_y$, as it is better adapted to the twisted operators.

Finally, we remark that since $y$ is purely radial,
\[
D_t \phi = \nab_t \phi = \pd_t \phi \text{,} \qquad D_A \phi = \snabla_A \phi = \pd_A \phi
\]
for scalar functions $\phi$.
Thus, we will use the above notations interchangeably whenever
convenient and whenever there is no risk of confusion.
Moreover, we will write
\[
D_X \phi = X^\al D_\al \phi
\]
to denote derivatives along a vector field $X$.

\subsection{Pointwise Hardy Inequalities}

Next, we establish a family of pointwise Hardy-type inequalities in terms of the twisted derivative operator $D$:

\begin{proposition}\label{P.Hardy}
For any $q \in \R$ and any $u \in C^1 ( \mc{C} )$, the following holds:
\begin{align}
\label{hardy} y^{ q - 1 } ( D_r u )^2 &\geq \frac{1}{4} ( 2 \kappa + q - 2 )^2 y^{ q - 3 } \cdot u^2 - ( n - 1 ) \left( \kappa + \frac{ q - 2 }{2} \right) y^{ q - 2 } r^{-1} \cdot u^2 \\
\notag &\qquad - \nabla^\beta \left[ \left( \kappa + \frac{ q - 2 }{2} \right) y^{ q - 2 } \nabla_\beta y \cdot u^2 \right] \text{.}
\end{align}
\end{proposition}
\begin{proof}

First, for any $p, b \in \R$, we have the inequality
\begin{align*}
0 &\leq ( y^p \cdot \nabla^\alpha y D_\alpha u + b y^{ p - 1 } \cdot u )^2 \\
&= y^{ 2 p } \cdot ( \nabla^\alpha y D_\alpha u )^2 + b^2 y^{ 2 p - 2 } \cdot u^2 + 2 b y^{ 2 p - 1 } \cdot u \nabla^\alpha y D_\alpha u \\
&= y^{ 2 p } \cdot ( D_r u )^2 + b ( b - 2 \kappa - 2 p + 1 ) y^{ 2 p - 2 } \cdot u^2 \\
&\qquad - b y^{ 2 p - 1 } \Box y \cdot u^2 + \nabla^\beta ( b y^{ 2 p - 1 } \nabla_\beta y \cdot u^2 ) \text{,}
\end{align*}
where we used \eqref{twisted} in the last step.
Setting $2 p = q - 1$, the above becomes
\begin{align*}
y^{ q - 1 } ( D_r u )^2 &\geq - b ( b - 2 \kappa - q + 2 ) y^{ q - 3 } \cdot u^2 + b y^{ q - 2 } \Box y \cdot u^2 \\
&\qquad - \nabla^\beta ( b y^{ q - 2 } \nabla_\beta y \cdot u^2 ) \text{.}
\end{align*}
Taking $b = \kappa + \frac{ q - 2 }{2}$ (which extremizes the above) yields \eqref{hardy}.
\end{proof}

\subsection{Boundary Asymptotics}

We conclude this section by discussing the precise boundary limits for our main results.
First, given $u \in C^1 ( \mc{C} )$, we define its Dirichlet and Neumann traces on $\Gamma$ with respect to $\Box_y$ (or equivalently, $\Box_\ka$) by
\begin{align}
\label{bddconds} \mc{D}_\kappa u: \Gamma \rightarrow \R \text{,} &\qquad \mc{D}_\kappa u: = \lim_{ r \nearrow 1 } ( y^{ -\kappa } u ) \text{,} \\
\notag \mc{N}_\kappa u: \Gamma \rightarrow \R \text{,} &\qquad \mc{N}_\kappa u := \lim_{ r \nearrow 1 } y^{ 2 \kappa } \partial_r ( y^{ - \kappa } u ) \text{.}
\end{align}
Note in particular that the formulas \eqref{bddconds} are directly inspired from \eqref{BCs}.

Now, the subsequent definition lists the main assumptions we will impose on boundary limits in our Carleman estimates and observability results:

\begin{definition} \label{admissible}
A function $u \in C^1 ( \mc{C} )$ is called \emph{boundary admissible} with respect to $\Box_y$ (or $\Box_\ka$) when the following conditions hold:
\begin{enumerate}[i)]
\item $\mc{N}_\kappa u$ exists and is finite.

\item The following Dirichlet limits hold for $u$:
\begin{equation}
\label{super_dirichlet} ( 1 - 2 \kappa ) \mc{D}_\kappa ( y^{ - 1 + 2 \kappa } u ) = - \mc{N}_\kappa u \text{,} \qquad \mc{D}_\kappa ( y^{ 2 \kappa } \partial_t u ) = 0 \text{.}
\end{equation}
\end{enumerate}
Here, the Dirichlet and Neumann limits are in an $L^2$-sense on $( -T, T ) \times \Sph^{ n - 1 }$.
\end{definition}

The main motivation for Definition \ref{admissible} is that \emph{it captures the expected boundary asymptotics for solutions of the equation $\Box_y u = 0$ that have vanishing Dirichlet data}.
(In particular, note that $u$ being boundary admissible implies $\mc{D}_\kappa u = 0$.)
To justify this statement, we must first recall some results from \cite{Warnick}.

For $u \in C^1 ( \mc{C} )$ and $\tau \in ( -T, T )$, we define the following twisted $H^1$-norms:
\begin{align}
\label{E1} E_1 [ u ] ( \tau ) &:= \int_{ \mc{C} \cap \{ t = \tau \} } ( | \partial_t u |^2 + | D_r u |^2 + | \snabla u |^2 + u^2 ) \text{,} \\
\bar{E}_1 [ u ] ( \tau ) &:= \int_{ \mc{C} \cap \{ t = \tau \} } ( | \partial_t u |^2 + | \bar{D}_r u |^2 + | \snabla u |^2 + u^2 ) \text{.}
\end{align}
Moreover, if $u \in C^2 ( \mc{C} )$ as well, then we define the twisted $H^2$-norm,
\begin{equation}
\label{E2} E_2 [ u ] ( \tau ) := \bar{E}_1 [ D_r u ] ( \tau ) + E_1 [ \partial_t u ] ( \tau ) + E_1 [ \snabla_t u ] ( \tau ) + E_1 [ u ] ( \tau ) \text{.}
\end{equation}
The results of \cite{Warnick} show that both $E_1 [ u ]$ and $E_2 [ u ]$ are natural energies associated with the operator $\Box_y$, in that their boundedness is propagated in time for solutions of $\Box_y u = 0$ with Dirichlet boundary conditions.

The following proposition shows that functions with uniformly bounded $E_2$-energy are boundary admissible, in the sense of Definition \ref{admissible}.
In particular, the preceding discussion then implies that boundary admissibility is achieved by sufficiently regular (in a twisted $H^2$-sense) solutions of the singular wave equation $\Box_y u = 0$, with Dirichlet boundary conditions.

\begin{proposition} \label{B.asymp}
Let $u \in C^2 ( \mc{C} )$, and assume that:
\begin{enumerate}[i)]
\item $\mc{D}_\kappa u = 0$.

\item $E_2 [ u ] ( \tau )$ is uniformly bounded for all $\tau \in ( -T, T )$.
\end{enumerate}
Then, $u$ is boundary admissible with respect to $\Box_y$, in the sense of Definition \ref{admissible}.
\end{proposition}

\begin{proof}
Fix $\tau \in ( -T, T )$ and $\omega \in \Sph^{ n - 1 }$, and let $0 < y_1 < y_0 \ll 1$.
Applying the fundamental theorem of calculus and integrating in $y$ yields
\begin{align*}
y^{ 2 \kappa } \partial_r ( y^{ - \kappa } u ) |_{ ( \tau, 1 - y_1, \omega ) } - y^{ 2 \kappa } \partial_r ( y^{ - \kappa } u ) |_{ ( \tau, 1 - y_0, \omega ) } 
&= \int_{ y_1 }^{ y_0 } y^\kappa \bar{D}_r ( D_r u ) |_{ ( \tau, 1 - y, \omega ) } dy \text{,}
\end{align*}
where we have described points in $\bar{\mc{C}}$ using polar $( t, r, \omega )$-coordinates.

We now integrate the above over $\Gamma = ( -T, T ) \times \Sph^{ n - 1 }$, and we let $y_1 \searrow 0$.
In particular, observe that for $\mc{N}_\kappa u$ to be finite, it suffices to show that
\[
I := \int_\Gamma \left[ \int_0^{ y_0 } y^\kappa \bar{D}_r ( D_r u ) |_{ ( \tau, 1 - y, \omega ) } dy \right]^2 d \tau d \omega < \infty \text{.}
\]
However, by H\"older's inequality and \eqref{kappa}, we have
\[
I \leq \int_\Gamma \left[ \int_0^{ y_0 } y^{ 2 \kappa } dy \int_0^{ y_0 } | \bar{D}_r ( D_r u ) |^2 |_{ ( \tau, 1 - y, \omega ) } dy \right] d \tau d \omega \lesssim \int_{ -T }^T E_2 [u] ( \tau ) \, d \tau \text{.}
\]
Thus, the assumptions of the proposition imply that $I$, and hence $\mc{N}_\kappa u$, is finite.

Next, to prove the first limit in \eqref{super_dirichlet}, it suffices to show that
\begin{equation}
\label{B.asymp_1} J_{ y_0 } := \int_\Gamma \left( y^{ -1 + \kappa } u |_{ ( \tau, 1 - y_0, \omega ) } + \frac{1}{ 1 + 2 \kappa } \mc{N}_\kappa u |_{ ( \tau, \omega ) } \right)^2 d \tau d \omega \rightarrow 0 \text{,}
\end{equation}
as $y_0 \searrow 0$.
Since $\mc{D}_k u = 0$, the fundamental theorem of calculus implies
\begin{align*}
J_{ y_0 } &= \int_\Gamma \left[ - y_0^{ -1 + 2 \kappa } \int_0^{ y_0 } y^{ - 2 \kappa } y^{ 2 \kappa } \partial_r ( y^{ -\kappa } u ) |_{ ( \tau, 1 - y, \omega ) } dy + \frac{1}{ 1 + 2 \kappa } \mc{N}_\kappa u |_{ ( \tau, \omega ) } \right]^2 d \tau d \omega \\
&= \int_\Gamma \left\{ y_0^{ -1 + 2 \kappa } \int_0^{ y_0 } y^{ - 2 \kappa } [ y^{ 2 \kappa } \partial_r ( y^{ -\kappa } u ) |_{ ( \tau, 1 - y, \omega ) } - \mc{N}_\kappa u |_{ ( \tau, \omega ) } ] dy \right\}^2 d \tau d \omega \text{.}
\end{align*}
Moreover, the Minkowski integral inequality yields
\begin{align*}
\sqrt{ J_{ y_0 } } &\leq y_0^{ -1 + 2 \kappa } \int_0^{ y_0 } y^{ -2 \kappa } \left\{ \int_\Gamma [ y^{ 2 \kappa } \partial_r ( y^{ -\kappa } u ) |_{ ( \tau, 1 - y, \omega ) } - \mc{N}_\kappa u |_{ ( \tau, \omega ) } ]^2 d \tau d \omega \right\}^\frac{1}{2} dy \\
&\lesssim \sup_{ 0 < y < y_0 } \left\{ \int_\Gamma [ y^{ 2 \kappa } \partial_r ( y^{ -\kappa } u ) |_{ ( \tau, 1 - y, \omega ) } - \mc{N}_\kappa u |_{ ( \tau, \omega ) } ]^2 d \tau d \omega \right\}^\frac{1}{2} \text{.}
\end{align*}
By the definition of $\mc{N}_\kappa u$, the right-hand side of the above converges to $0$ when $y_0 \searrow 0$.
This implies \eqref{B.asymp_1}, and hence the first part of \eqref{super_dirichlet}.

For the remaining limit in \eqref{super_dirichlet}, we first claim that $\mc{D}_\kappa ( \partial_t u )$ exists and is finite.
This argument is analogous to the first part of the proof.
Note that since
\[
y^{ - \kappa } \partial_t u |_{ ( \tau, 1 - y_1, \omega ) } - y^{ - \kappa } \partial_t u |_{ ( \tau, 1 - y_0, \omega ) } = \int_{ y_1 }^{ y_0 } y^{ -\kappa } D_r \partial_t u |_{ ( \tau, 1 - y, \omega ) } dy \text{,}
\]
then the claim immediately follows from the fact that
\[
\int_\Gamma \left[ \int_0^{ y_0 } y^{ - \kappa } D_r \partial_t u |_{ ( \tau, 1 - y, \omega ) } dy \right]^2 d \tau d \omega \lesssim \int_{ -T }^T E_2 [u] ( \tau ) \, d \tau < \infty \text{.}
\]

Moreover, to determine $\mc{D}_\kappa ( \partial_t u )$, we see that
for any test function $\varphi \in C^\infty_0 ( \Gamma )$,
\[
\int_\Gamma \mc{D}_\kappa ( \partial_t u ) \cdot \varphi = - \lim_{ y \searrow 0 } \int_\Gamma y^{ - \kappa } u |_{ r = 1 - y } \cdot \partial_t \varphi = - \int_\Gamma \mc{D}_\kappa u \cdot \partial_t \varphi = 0 \text{.}
\]
It then follows that $\mc{D}_\kappa ( \partial_t u ) = 0$.

Finally, to prove the second limit of \eqref{super_dirichlet}, it suffices to show
\begin{equation}
\label{B.asymp_2} K_{ y_0 } := \int_\Gamma ( y^{ - \frac{1}{2} } \partial_t u )^2 |_{ ( \tau, 1 - y_0, \omega ) } d \tau d \omega \rightarrow 0 \text{,} \qquad y_0 \searrow 0 \text{.}
\end{equation}
Using that $\mc{D}_\ka ( \partial_t u ) = 0$ along with the fundamental theorem of calculus yields
\begin{align*}
K_{ y_0 } &= \int_\Gamma \left[ y_0^{ -\frac{1}{2} + \kappa } \int_0^{ y_0 } y^{ - \kappa } D_r \partial_t u |_{ ( \tau, 1 - y, \omega ) } dy \right]^2 d \tau d \omega \\
&\leq y_0^{ -1 + 2 \kappa } \int_\Gamma \left[ \int_0^{ y_0 } y^{ - 2 \kappa } d y \int_0^{ y_0 } ( D_r \partial_t u )^2 |_{ ( \tau, 1 - y, \omega ) } dy \right] d \tau d \omega \\
&\lesssim \int_0^{ y_0 } \int_\Gamma ( D_r \partial_t u )^2 |_{ ( \tau, 1 - y, \omega ) } d \tau d \omega dy \text{.}
\end{align*}
The integral on the right-hand side is (the time integral of) $E_2 [ u ] ( \tau )$, restricted to the region $1 - y_0 < r < 1$.
Since $E_2 [ u ] ( \tau )$ is uniformly bounded, it follows that $K_{ y_0 }$ indeed converges to zero as $y_0 \searrow 0$, completing the proof.
\end{proof}

\begin{remark}
From the intuitions of \cite{Gueye}, one may conjecture that Proposition \ref{B.asymp} could be further strengthened, with the boundedness assumption on $E_2 [ u ]$ replaced by a sharp boundedness condition on an appropriate fractional $H^{ 1 + \kappa }$-norm.
However, we will not pursue this question in the present paper.
\end{remark}

\section{Multiplier Inequalities} \label{S.multipliers}

In this section, we derive some multiplier identities and inequalities, which form the foundations of the proof of the main Carleman estimates, Theorem \ref{T.Carleman}.
As mentioned before, these can be viewed as extensions to singular wave operators of the classical Morawetz inequality for wave equations.

In what follows, we fix $0 < \varepsilon \ll 1$, and we define the cylindrical region
\begin{equation}
\label{C_eps} \mc{C}_\varepsilon := ( -T, T ) \times \{ \varepsilon < r < 1 - \varepsilon \} \text{.}
\end{equation}
Moreover, let $\Gamma_\varepsilon$ denote the timelike boundary of $\mc{C}_\varepsilon$:
\begin{equation}
\label{Gamma_eps} \Gamma_\varepsilon := \Ga_\ep^- \cup \Ga_\ep^+ := [ ( -T, T ) \times \{ r = \varepsilon \} ] \cup [ ( -T, T ) \times\{ r = 1 - \varepsilon \} ] \text{.}
\end{equation}
We also let $\nu$ denote the unit outward-pointing ($g$-)normal vector field on $\Gamma_\varepsilon$.

Finally, we fix a constant $c > 0$, and we define the functions
\begin{equation}
\label{f,z} f := - \frac{1}{ 1 + 2 \ka } \cdot y^{ 1 + 2\ka } - c t^2 \text{,} \qquad z := - 4 c \text{,}
\end{equation}
which will be used to construct the multiplier for our upcoming inequalities.

\subsection{A Preliminary Identity}

We begin by deriving a preliminary form of our multiplier identity, for which the multiplier is defined using $f$ and $z$:

\begin{proposition} \label{T.mult_general}
Let $u \in C^\infty ( \mc{C} )$, and assume $u$ is supported on $\mc{C} \cap \{ |t| < T - \de \}$ for some $0 < \delta \ll 1$.
Then, we have the identity,
\begin{align}
\label{gralmult} - \int_{ \mc{C}_\varepsilon } \Box_y u \cdot S_{ f, z } u &= \int_{ \mc{C}_\varepsilon } ( \nabla^{ \alpha \beta } f + z \cdot g^{ \alpha \beta } ) D_\alpha u D_\beta u + \int_{ \mc{C}_\varepsilon } \mc{A}_{ f, z } \cdot u^2 \\
\notag &\qquad - \int_{ \Gamma_\varepsilon } S_{ f, z } u \cdot D_\nu u + \frac{1}{2} \int_{ \Gamma_\varepsilon } \nabla_\nu f \cdot D_\beta u
 D^\beta u \\
\notag &\qquad + \frac{1}{2} \int_{ \Gamma_\varepsilon } \nabla_\nu w_{ f, z } \cdot u^2 \text{,}
\end{align}
for any $0 < \varepsilon \ll 1$, where
\begin{align}
\label{gral_wAS} w_{ f, z } &:= \frac{1}{2} \left( \Box f + \frac{ 2 \kappa }{y} \nabla_\alpha y \nabla^\alpha f \right) + z \text{,} \\
\notag \mc{A}_{ f, z } &:= - \frac{1}{2} \left( \Box w_{ f, z } + \frac{ 2 \kappa }{y} \nabla_\alpha y \nabla^\alpha w_{ f, z } \right) \text{,} \\
\notag S_{ f, z } &:= \nabla^\alpha f \cdot D_\alpha + w_{ f, z } \text{.}
\end{align}
\end{proposition}

\begin{proof}
Integrating the left-hand side of \eqref{gralmult} by parts twice reveals that
\begin{align*}
- \int_{ \mc{C}_\varepsilon } \Box_y u \cdot \nabla^\alpha f D_\alpha u &= \int_{ \mc{C}_\varepsilon } D_\beta u \cdot D^\beta ( \nabla^\alpha f D_\alpha u ) - \int_{ \Gamma_\varepsilon } \nabla^\alpha f D_\alpha u \cdot D_\nu u \\
&= \int_{ \mc{C}_\varepsilon } \nabla^{ \alpha \beta } f \cdot
  D_\alpha u D_\beta u + \int_{ \mc{C}_\varepsilon } \nabla^\alpha f \cdot D_\beta u D_\alpha{}^\beta u \\
&\qquad  - \int_{ \Gamma_\varepsilon } \nabla^\alpha f D_\alpha u \cdot D_\nu u \\
&= \int_{ \mc{C}_\varepsilon } \nabla^{ \alpha \beta } f \cdot D_\alpha u D_\beta u + \frac{1}{2} \int_{ \mc{C}_\varepsilon } \nabla^\alpha f \cdot \nabla_\alpha ( D_\beta u D^\beta u ) \\
&\qquad - \int_{ \mc{C}_\varepsilon } \frac{ \kappa }{y} \nabla_\alpha y \nabla^\alpha f \cdot D_\beta u D^\beta u - \int_{ \Gamma_\varepsilon } \nabla^\alpha f D_\alpha u \cdot D_\nu u \\
&= \int_{ \mc{C}_\varepsilon } \left[ \nabla^{ \alpha \beta } f  - \frac{1}{2}\left( \Box f + \frac{ 2 \kappa }{y} \nabla_\alpha y \nabla^\alpha f \right) g^{\al\be} \right] \cdot D_\alpha u D_\beta u \\
&\qquad - \int_{ \Gamma_\varepsilon } \nabla^\alpha f D_\alpha u \cdot D_\nu u + \frac{1}{2} \int_{ \Gamma_\varepsilon } \nabla_\nu f \cdot D_\beta u D^\beta u \text{,}
\end{align*}
where in the above steps, we also applied the identities \eqref{twisted}, \eqref{prod_rule}, \eqref{Box_y}, as well as the observation that $\bar{D}$ is the adjoint of $D$.

A similar set of computations also yields
\begin{align*}
- \int_{ \mc{C}_\varepsilon } \Box_y u \cdot w_{ f, z } u &= \int_{ \mc{C}_\varepsilon } D^\alpha u D_\alpha ( w_{ f, z } u ) - \int_{ \Gamma_\varepsilon } w_{ f, z } \cdot u D_\nu u \\
&= \int_{ \mc{C}_\varepsilon } \nabla_\alpha w_{ f, z } \cdot u D^\alpha u + \int_{ \mc{C}_\varepsilon } w_{ f, z } \cdot D^\alpha u D_\alpha u - \int_{ \Gamma_\varepsilon } w_{ f, z } \cdot u D_\nu u \\
&= \int_{ \mc{C}_\varepsilon } w_{ f, z } \cdot D^\alpha u D_\alpha u + \frac{1}{2} \int_{ \mc{C}_\varepsilon } \nabla_\alpha w_{ f, z } \cdot \nabla^\alpha ( u^2 ) \\
&\qquad - \int_{ \mc{C}_\varepsilon } \frac{ \kappa }{y} \nabla^\alpha y \nabla_\alpha w_{ f, z } \cdot u^2 - \int_{ \Gamma_\varepsilon } w_{ f, z } \cdot u D_\nu u \\
&= \int_{ \mc{C}_\varepsilon } w_{ f, z } \cdot D^\alpha u D_\alpha u - \frac{1}{2} \int_{ \mc{C}_\varepsilon } \left( \Box w_{ f, z } + \frac{ 2 \kappa }{y} \nabla^\alpha y \nabla_\alpha w_{ f, z } \right) \cdot u^2 \\
&\qquad - \int_{ \Gamma_\varepsilon } w_{ f, z } \cdot u D_\nu u + \frac{1}{2} \int_{ \Gamma_\varepsilon } \nabla_\nu w_{ f, z } \cdot u^2 \text{.}
\end{align*}
Adding the above two identities results in \eqref{gralmult}.
\end{proof}

\subsection{Computations for $f$ and $z$}

In the following proposition, we collect some computations involving the functions $f$ and $z$ that will be useful later on.

\begin{proposition} \label{T.f,z}
$f$, $w_{ f, z }$, and $\mc{A}_{ f, z }$ (defined as in \eqref{f,z} and \eqref{gral_wAS}) satisfy
\begin{align}
\label{wAS} \nabla_{ \alpha \beta } f &= y^{ 2 \kappa } \cdot \nabla_{ \alpha \beta } r - 2 \kappa y^{ 2 \kappa - 1 } \cdot \nabla_\alpha r \nabla_\beta r - 2 c \cdot \nabla_\alpha t \nabla_\beta t \text{,} \\
\notag w_{ f, z } &= - 2 \kappa \cdot y^{ 2 \kappa - 1 } + \frac{1}{2} ( n - 1 ) \cdot y^{ 2 \kappa } r^{-1} - 3 c \,,\\
\notag \mc{A}_{ f, z } &= 2 \kappa ( 2 \kappa - 1 )^2 \cdot y^{ 2 \kappa - 3 } - \frac{1}{2} ( n - 1 ) \kappa ( 8 \kappa - 3 ) \cdot y^{ 2 \kappa - 2 } r^{-1} \\
\notag &\qquad + \frac{1}{2} ( n - 1 ) ( n - 4 ) \kappa \cdot y^{ 2 \kappa -
  1 } r^{-2} + \frac{1}{4} ( n - 1 ) ( n - 3 ) \cdot y^{ 2 \kappa } r^{-3} \text{.}
\end{align}
\end{proposition}

\begin{proof}
First, we fix $q \in \R \setminus \{ -1 \}$, and we let
\begin{equation}
\label{fq} f_q := - \frac{ y^{ 1 + q } }{ 1 + q } \text{.}
\end{equation}
Note that $f_q$ satisfies
\begin{align}
\label{fq_deriv} \nabla_\alpha f_q &= - y^q \cdot \nabla_\alpha y \text{,} \\
\notag \nabla_{ \alpha \beta } f_q &= - y^q \cdot \nabla_{ \alpha \beta } y - q y^{ q - 1 } \cdot \nabla_\alpha y \nabla_\beta y \text{,} \\
\notag \Box f_q &= - y^q \cdot \Box y - q y^{ q - 1 } \cdot \nabla^\alpha y \nabla_\alpha y \text{,} \\
\notag \frac{ 2 \kappa }{y} \cdot \nabla^\alpha y \nabla_\alpha f_q &= - 2 \kappa y^{ q - 1 } \cdot \nabla^\alpha y \nabla_\alpha y \text{.}
\end{align}

Next, using the notations from \eqref{gral_wAS}, along with \eqref{y_id} and \eqref{fq_deriv}, we have
\begin{align}
\label{wq} w_{ f_q, 0 } &= - \frac{1}{2} y^q \cdot \Box y - \left( \kappa + \frac{ q }{2} \right) y^{ q - 1 } \cdot \nabla^\alpha y \nabla_\alpha y \\
\notag &= - \left( \kappa + \frac{ q }{2} \right) \cdot y^{ q - 1 } + \frac{ n - 1 }{2} \cdot y^q r^{-1} \text{.}
\end{align}
Moreover, further differentiating \eqref{wq} and again using \eqref{y_id}, we see that
\begin{align*}
\Box w_{ f_q, 0 } &= - \frac{1}{2} ( q + 2 \kappa ) ( q - 1 ) ( q - 2 ) y^{ q - 3 } \cdot ( \nabla^\alpha y\nabla_\alpha y )^2 \\
&\qquad - ( q - 1 ) [( q + \kappa )  \Box y \nabla^\alpha y \nabla_\alpha y + 2 ( q + 2 \kappa ) \nabla^{ \alpha \beta } y \nabla_\alpha y \nabla_\beta y]\cdot y^{ q - 2 } \\
&\qquad - 2 ( q + \kappa ) y^{ q - 1 } \cdot \nabla^\alpha y \nabla_\alpha ( \Box y ) - ( q + 2 \kappa ) y^{ q - 1 } \cdot \nabla^{ \alpha \beta } y \nabla_{ \alpha \beta } y \\
&\qquad - \frac{1}{2} q y^{ q - 1 } \cdot ( \Box y)^2 - \frac{1}{2} y^q \cdot \Box^2 y \text{,} \\
\frac{ 2 \kappa }{y} \nabla^\alpha y \nabla_\alpha w_{ f_q, 0 } &= - \kappa ( q + 2 \kappa ) ( q - 1 ) y^{ q - 3 } \cdot ( \nabla^\alpha y \nabla_\alpha y )^2 - \kappa q y^{ q - 2 } \cdot \Box y \nabla^\alpha y \nabla_\alpha y \\
&\qquad - 2 \kappa ( q + 2 \kappa ) y^{ q - 2 } \cdot \nabla^{ \alpha \beta } y \nabla_\alpha y \nabla_\beta y - \kappa y^{ q - 1 } \cdot \nabla^\alpha y \nabla_\alpha ( \Box y ) \text{.}
\end{align*}
We can then use the above to compute the coefficient $\mc{A}_{ f_q, 0 }$:
\begin{align}
\label{Aq} \mc{A}_{ f_q, 0 } &= \frac{1}{4} ( q + 2 \kappa ) ( q + 2 \kappa - 2 ) ( q - 1 ) y^{ q - 3 } \cdot ( \nabla^\alpha y \nabla_\alpha y )^2 \\
\notag &\qquad + \frac{1}{2} ( q^2 - q + 2 \kappa q - \kappa ) y^{ q - 2 } \cdot \Box y \nabla^\alpha y \nabla_\alpha y \\
\notag &\qquad + ( q + 2 \kappa ) ( q + \kappa - 1 ) y^{ q - 2 } \cdot \nabla^{ \alpha \beta } y \nabla_\alpha y \nabla_\beta y \\
\notag &\qquad   + \frac{1}{2} ( 2 q + 3 \kappa ) y^{ q - 1 } \cdot \nabla^\alpha y \nabla_\alpha ( \Box y ) + \frac{1}{2} ( q + 2 \kappa ) y^{ q - 1 } \cdot \nabla^{ \alpha \beta } y \nabla_{ \alpha \beta } y \\
\notag &\qquad + \frac{1}{4} q y^{ q - 1 } \cdot ( \Box y )^2 + \frac{1}{4} y^q \cdot \Box^2 y \\
\notag &= \frac{1}{4} ( q + 2 \kappa ) ( q + 2 \kappa - 2 ) ( q - 1 ) \cdot y^{ q - 3 } \\
\notag &\qquad - \frac{1}{2} ( n - 1 ) ( q^2 - q + 2 \kappa q - \kappa ) \cdot y^{ q - 2 } r^{-1} \\
\notag &\qquad + \frac{1}{4} ( n - 1 ) [ q ( n - 3 ) - 2 \kappa ] \cdot y^{ q - 1 } r^{-2} + \frac{1}{4} ( n - 1 ) ( n - 3 ) \cdot y^q r^{-3} \text{.}
\end{align}

Notice from \eqref{f,z} and \eqref{fq} that we can write
\[
f = f_{ 2 \kappa } - c t^2 \text{,}
\]
Thus, substituting $q = 2 \kappa$ in \eqref{fq}, we see that the Hessian of $f$ satisfies
\begin{align*}
\nabla_{ \alpha \beta } f &= \nabla_{ \alpha \beta } f_{ 2 \kappa } - c \nabla_{ \alpha \beta } t^2 \\
&= y^{ 2 \kappa } \cdot \nabla_{ \alpha \beta } r - 2 \kappa y^{ 2 \kappa - 1 } \cdot \nabla_\alpha r \nabla_\beta r - 2 c \nabla_\alpha t \nabla_\beta t \text{,}
\end{align*}
which is precisely the first part of \eqref{wAS}.

Moreover, noting that
\[
w_{ - c t^2, 0 } = c \text{,}
\]
then we also have
\begin{align*}
w_{ f, z } &= w_{ f_{ 2 \kappa }, 0 } + w_{ - c t^2, 0 } + z \\
&= - 2 \kappa \cdot y^{ 2 \kappa - 1 } + \frac{1}{2} ( n - 1 ) \cdot y^{ 2 \kappa } r^{-1} - 3 c \,,
\end{align*}
which gives the second equation in \eqref{wAS}.
Finally, noting that
\[
\mc{A}_{ - c t^2, 0 } = 0 \text{,} \qquad - \frac{1}{2} \left( \Box z + \frac{ 2 \kappa }{y} \cdot \nabla^\alpha y \nabla_\alpha z \right) = 0 \text{,}
\]
we obtain, with the help of \eqref{y_id}, the last equation of \eqref{wAS}:
\begin{align*}
\mc{A}_{ f, z } &= \mc{A}_{ f_{ 2 \kappa }, 0 } + \mc{A}_{ - c t^2, 0 } - \frac{1}{2} \left( \Box z + \frac{ 2 \kappa }{y} \cdot \nabla^\alpha y \nabla_\alpha z \right) \\
&= 2 \kappa ( 2 \kappa - 1 )^2 y^{ 2 \kappa - 3 } \cdot ( \nabla^\alpha y \nabla_\alpha y )^2 + \frac{1}{2} \kappa ( 8 \kappa - 3 ) y^{ 2 \kappa - 2 } \cdot \Box y \nabla^\alpha y \nabla_\alpha y \\
&\qquad + 4 \kappa ( 3 \kappa - 1 ) y^{ 2 \kappa - 2 } \cdot \nabla^{ \alpha \beta } y \nabla_\alpha y \nabla_\beta y + \frac{7}{2} \kappa y^{ 2 \kappa - 1 } \cdot \nabla^\alpha y \nabla_\alpha ( \Box y ) \\
&\qquad + 2 \kappa y^{ 2 \kappa - 1 } \cdot \nabla^{ \alpha \beta } y \nabla_{ \alpha \beta } y + \frac{1}{2} \kappa y^{ 2 \kappa - 1 } \cdot ( \Box y )^2 + \frac{1}{4} y^{ 2 \kappa } \cdot \Box^2 y \\
&= 2 \kappa ( 2 \kappa - 1 )^2 \cdot y^{ 2 \kappa - 3 } - \frac{1}{2} ( n - 1 ) \kappa ( 8 \kappa - 3 ) \cdot y^{ 2 \kappa - 2 } r^{-1} \\
&\qquad + \frac{1}{2} ( n - 1 ) ( n - 4 ) \kappa \cdot y^{ 2 \kappa - 1 } r^{-2} + \frac{1}{4} ( n - 1 ) ( n - 3 ) \cdot y^{ 2 \kappa } r^{-3} \text{.} \qedhere
\end{align*}
\end{proof}

\subsection{The Main Inequality}

We conclude this section with the multiplier inequality that will be used to prove our main Carleman estimate:

\begin{proposition} \label{T.multineq}
Let $f$ and $z$ be as in \eqref{f,z}, and let $u \in C^\infty ( \mc{C} )$ be supported on $\mc{C} \cap \{ |t| < T - \de \}$ for some $0 < \delta \ll 1$.
Then, we have the inequality
\begin{align}
\label{multineq} - \int_{ \mc{C}_\varepsilon } \Box_y u \cdot S_{ f, z } u &\geq \int_{ \mc{C}_\varepsilon } [ ( 1 - 4 c ) \cdot | \snabla u |^2 + 2 c \cdot ( \pd_t u )^2 - 4 c \cdot ( D_r u )^2 ] \\
\notag &\qquad - \frac{1}{2} ( n - 1 ) \kappa \int_{ \mc{C}_\varepsilon } y^{ 2 \kappa - 2 } r^{-2} [ r - ( n - 4 ) y ] \cdot u^2 \\
\notag &\qquad + \frac{1}{4} ( n - 1 ) ( n - 3 ) \int_{ \mc{C}_\varepsilon } y^{ 2 \kappa } r^{-3} \cdot u^2 - \int_{ \Gamma_\varepsilon } S_{ f,z } u \cdot D_\nu u \\
\notag &\qquad + \frac{1}{2} \int_{ \Gamma_\varepsilon } \nabla_\nu f \cdot D_\beta u D^\beta u + \frac{1}{2} \int_{ \Gamma_\varepsilon } \nabla_\nu w_{ f, z } \cdot u^2 \\
\notag &\qquad + 2 \kappa ( 2 \kappa - 1 ) \int_{ \Gamma_\varepsilon } y^{ 2 \kappa - 2 } \nabla_\nu y \cdot u^2 \text{,}
\end{align}
for any $0 < \varepsilon \ll 1$, where $w_{ f, z }$ and $S_{ f, z }$ are defined as in \eqref{gral_wAS}.
\end{proposition}

\begin{proof}
Applying the multiplier identity \eqref{gralmult}, with $f$ and $z$ from \eqref{f,z}, and recalling the formulas \eqref{wAS} for $\nabla^2 f$, $w_{ f, z }, $ and $\mc{A}_{ f, z }$, we obtain that
\[
I := - \int_{ \mc{C}_\varepsilon } \Box_y u \cdot S_{ f, z } u
\]
satisfies the identity
\begin{align}
\label{multineq_1} I &= \int_{\mc{C}_\varepsilon} ( y^{ 2 \kappa } \nabla^{ \alpha \beta } r - 2 \kappa y^{ -1 + 2 \kappa } \nabla^\alpha r \nabla^\beta r - 2 c \nabla^\alpha t \nabla^\beta t - 4 c g^{\alpha \beta } ) D_\alpha u D_\beta u \\
\notag &\qquad + 2 \kappa ( 2 \kappa - 1 )^2 \int_{ \mc{C}_\varepsilon } y^{ 2 \kappa - 3 } u^2 - \frac{1}{2} ( n - 1 ) \kappa ( 8 \kappa - 3 ) \int_{ \mc{C}_\varepsilon } y^{ 2 \kappa - 2 } r^{-1} u^2 \\
\notag &\qquad + \frac{1}{2} ( n - 1 ) ( n - 4 ) \kappa \int_{ \mc{C}_\varepsilon } y^{ 2 \kappa - 1 } r^{-2} u^2 + \frac{1}{4} ( n - 1 ) ( n - 3 ) \int_{ \mc{C}_\varepsilon } y^{ 2 \kappa } r^{-3} u^2 \\
\notag &\qquad - \int_{ \Gamma_\varepsilon } S_{ f, z } u \cdot D_\nu u + \frac{1}{2} \int_{ \Gamma_\varepsilon } \nabla_\nu f \cdot D_\beta u D^\beta u + \frac{1}{2} \int_{ \Gamma_\varepsilon } \nabla_\nu w_{ f, z } \cdot u^2 \text{.}
\end{align}

For the first-order terms in the multiplier identity, we notice that
\[
\nabla^{ \alpha \beta } r \cdot D_\alpha u D_\beta u = r^{-1} |\snabla u |^2 \text{,} \qquad | \snabla u |^2 = g^{AB} \snabla_A u \snabla_B u \text{,}
\]
and we hence expand
\begin{align}
\label{multineq_2} &( y^{ 2 \kappa } \cdot \nabla^{ \alpha \beta } r - 2 \kappa y^{ -1 + 2 \kappa } \nabla^\alpha r \nabla^\beta r - 2 c \cdot \nabla^\alpha t \nabla^\beta t - 4 c \cdot g^{ \alpha \beta } ) D_\alpha u D_\beta u \\
\notag &\quad \geq - 2 \kappa y^{ - 1 + 2 \kappa } ( D_r u )^2 + ( y^{ 2 \kappa } r^{-1} - 4 c ) | \snabla u |^2 + 2 c ( \pd_t u )^2 - 4 c ( D_r u )^2 \\
\notag &\quad \geq - 2 \kappa y^{ - 1 + 2 \kappa } ( D_r u )^2 + ( 1 - 4 c ) | \snabla u |^2 + 2 c ( \pd_t u )^2 - 4 c ( D_r u )^2 \text{.}
\end{align}
Moreover, applying the Hardy inequality \eqref{hardy}, with $q = 2 \kappa$, yields
\begin{align}
\label{multineq_3} - 2 \kappa y^{ 2 \kappa - 1 } ( D_r u )^2 &\geq - 2 \kappa ( 2 \kappa - 1 )^2 y^{ 2 \kappa - 3 } u^2 + ( n - 1 ) 2 \kappa ( 2 \kappa - 1 ) y^{ 2 \kappa - 2 } r^{-1} u^2 \\
\notag &\qquad + 2 \kappa ( 2 \kappa - 1 )\nabla^\beta (y^{ 2 \kappa - 2 } \nabla_\beta y \cdot u^2 ) \text{.}
\end{align}

The desired inequality \eqref{multineq} now follows by combining \eqref{multineq_1}--\eqref{multineq_3} and applying the divergence theorem to the last term in \eqref{multineq_3}.
\end{proof}

\begin{remark} \label{rmk.pseudoconvex}
We note that the pseudo-convexity of the function $f$ (with respect to $\Box$) is implicit from the proof of Proposition \ref{T.multineq}.
While this was not shown directly, one can, with a few more computations, observe that the quantity $\nabla^2 f + z \cdot g$ is positive-definite when restricted to the directions tangent to the level sets of $f$.
Of course, this is a necessary condition for our upcoming Carleman estimates.
\end{remark}

\section{The Carleman Estimates} \label{S.Carleman}

In this section, we apply the preceding multiplier inequality to obtain our main Carleman estimates.
The precise statement of our estimates is the following:

\begin{theorem} \label{T.Carleman}
Assume $n \neq 2$, and fix $-\frac{1}{2} < \ka < 0$.
Also, let $u \in C^\infty ( \mc{C} )$ satisfy:
\begin{enumerate}[i)]
\item $u$ is boundary admissible (see Definition \ref{admissible}).

\item $u$ is supported on $\mc{C} \cap \{ |t| < T - \delta \}$ for some $\delta > 0$.
\end{enumerate}
Then, there exists some sufficiently large $\la_0 > 0$, depending only on $n$ and $\kappa$, such that the following Carleman inequality holds for all $\lambda \geq \lambda_0$:
\begin{align}
\label{Carleman} &\la \int_{\Ga} e^{ 2 \la f } ( \mc{N}_\ka u )^2 + \int_{ \mc{C} } e^{ 2 \la f } ( \Box_\ka u )^2 \\
\notag &\quad \geq C_0 \la \int_{\mc{C} } e^{ 2 \la f } [ ( \pd_t u )^2 + | \snabla u |^2 + ( D_r u )^2 ] + C_0 \lambda^3 \int_{ \mc{C} } e^{ 2 \la f } y^{ 6 \ka - 1 } u^2 \\
\notag &\quad\qquad + C_0 \lambda \cdot \begin{cases}
  \int_{ \mc{C} } e^{ 2 \la f } y^{ 2 \kappa - 2 } r^{-3} u^2 & \quad n \geq 4 \\
  \int_{ \mc{C} } e^{ 2 \la f } y^{ 2 \kappa - 2 } r^{-2} u^2 & \quad n = 3 \\
  0 & \quad n = 1
\end{cases} \text{.}
\end{align}
where the constant $C_0 > 0$ depends on $n$ and $\kappa$, where
\[
f = - \frac{1}{ 1 + 2 \ka } \cdot y^{ 1 + 2 \ka } - c t^2 \,,
\]
as in~\eqref{f,z},
and where the constant $c$ satisfies
\begin{equation}
\label{eqc} 0 < c < \frac{1}{5} \text{,} \qquad
\begin{cases}
  c \leq \frac{ 1 }{ 4 \sqrt{3} \cdot T} & \quad n \geq 4 \\
  c \leq \min \left\{ \frac{1}{4 \sqrt{15} \cdot T}, \frac{|\ka|}{120} \right\} & \quad n = 3 \\
  c \leq \frac{ 1}{ 4 \sqrt{15} \cdot T} & \quad n = 1
\end{cases} \text{.}
\end{equation}
\end{theorem}

The proof of Theorem~\ref{T.Carleman} is carried out in remainder of this section.

\begin{remark}
We note that parts of this proof will treat the cases $n = 1$, $n = 3$, and $n \geq 4$ separately.
This accounts for the difference in the assumptions for $c$ in \eqref{eqc}, which will affect the required timespan in our upcoming observability inequalities.
\end{remark}

\subsection{The Conjugated Inequality}

From here on, let us assume the hypotheses of Theorem \ref{T.Carleman}.
Let us also suppose that $\lambda_0$ is sufficiently large, with its precise value depending only on $n$ and $\kappa$.
In addition, we define the following:
\begin{equation}
\label{eqv} v := e^{ \lambda f } u \text{,} \qquad \mc{L} v := e^{ \lambda f } \Box_y ( e^{ - \lambda f } v ) \text{.}
\end{equation}

The objective of this subsection is to establish the following inequality for $v$:

\begin{lemma} \label{L.conjest}
For any $\lambda \geq \lambda_0$, we have the inequality
\begin{align}
\label{conjest} \frac{1}{ 4 \lambda } \int_{ \mc{C}_\varepsilon } ( \mc{L} v )^2 &\geq \frac c 2 \int_{ \mc{C}_\varepsilon } \left[( \pd_t v )^2+ | \snabla v |^2  +  ( D_r v )^2 \right] - \frac{1}{2} \kappa \lambda^2 \int_{ \mc{C}_\varepsilon } y^{ 6  \kappa - 1 } v^2 \\
\notag &\qquad + \frac{1}{2} \int_{ \Gamma_\varepsilon } \nabla_\nu f \cdot D_\beta v D^\beta v - \int_{ \Gamma_\varepsilon } S_{ f, z } v \cdot D_\nu v \\
\notag &\qquad - \frac{1}{2} \int_{ \Gamma_\varepsilon } [ \lambda^2 ( y^{ 4 \kappa } - 4 c^2 t^2 ) -8 c\la ] \nabla_\nu f  \cdot v^2 \\
\notag &\qquad + \frac{1}{2} \int_{ \Gamma_\varepsilon } \nabla_\nu w_{ f, z } \cdot v^2 + 2 \kappa ( 2 \kappa - 1 ) \int_{ \Gamma_\varepsilon } y^{ 2 \kappa - 2 } \nabla_\nu y \cdot v^2 \\
\notag &\qquad + \begin{cases}
  c_1 \int_{ \mc{C}_\varepsilon } y^{ 2 \kappa - 2 } r^{-3} \cdot v^2 & \quad n \geq 4 \\
  c_1 \int_{ \mc{C}_\varepsilon } y^{ 2 \kappa - 2 } r^{-2} \cdot v^2 + c_2 \int_{ \Gamma_\varepsilon } y^{ 4 \kappa - 1 } \nabla_\nu y \cdot v^2 & \quad n = 3 \\
  c_2 \int_{ \Gamma_\varepsilon } y^{ 4 \kappa - 1 } \nabla_\nu y \cdot v^2 & \quad n = 1
\end{cases} \,,
\end{align}
where $S_{ f, z }$ and $w_{ f, z }$ are defined as in \eqref{gral_wAS} and \eqref{wAS}, where the constant $c_1 > 0$ depends on $n$ and $\kappa$, and where the constant $c_2 > 0$ depends on $n$.
\end{lemma}
\begin{proof}
First, observe that by \eqref{twisted}--\eqref{Box_y}, we can expand $\mc{L} v$ as follows:
\begin{align}
\label{conjest_01} \mc{L} v &= e^{ \lambda f } \bar{D}^\alpha D_\alpha ( e^{ - \lambda f } v ) \\
\notag &= e^{ \lambda f } \bar{D}^\alpha ( e^{ - \lambda f } D_\alpha v ) - \lambda e^{ \lambda f } \bar{D}^\alpha ( e^{ - \lambda f } \nabla_\alpha f \cdot v ) \\
\notag &= \Box_y v - \lambda \nabla^\alpha f ( D_\alpha \psi + \bar{D}_\alpha v ) - \lambda \Box f \cdot v + \lambda^2 \nabla^\alpha f \nabla_\alpha f \cdot v \\
\notag &= \Box_y v - 2 \lambda S_{ f, z } v + \mc{A}_0 v \text{,}
\end{align}
where $\mc{A}_0$ is given by
\begin{equation}
\label{conjest_A0} \mc{A}_0 := \lambda^2 \nabla^\alpha f \nabla_\alpha f + 2 \la z = \lambda^2 ( y^{ 4 \kappa } - 4 c^2 t^2 ) - 8 c \la \text{.}
\end{equation}
Multiplying \eqref{conjest_01} by $S_{ f, z } v$ yields
\begin{equation}
\label{conjest_10} - \mc{L} v S_{ f, z } v = - \Box_y v S_{ f, z } v + 2 \lambda ( S_{ f, z } v )^2 - \mc{A}_0 \cdot v S_{ f, z } v \text{.}
\end{equation}

For the last term, we apply \eqref{twisted} and the product rule:
\begin{align}
\label{conjest_11} - \mc{A}_0 \cdot v S_{ f, z } v &= - \mc{A}_0 \cdot v ( \nabla^\alpha f D_\alpha v + w_{ f, z } v ) \\
\notag &= - \mc{A}_0 \cdot \left[ \frac{1}{2} \nabla^\alpha f \nabla_\alpha ( v^2 ) - \frac{ \kappa }{y} \nabla^\alpha f \nabla_\alpha y \cdot v^2 + w_{ f, z } v^2 \right] \\
\notag &= - \nabla^\alpha \left( \frac{1}{2} \mc{A}_0 \nabla_\alpha f \cdot v^2 \right) + \frac{1}{2} \nabla^\alpha f \nabla_\alpha \mc{A}_0 \cdot v^2 - z \mc{A}_0 \cdot v^2 \text{.}
\end{align}
Moreover, recalling \eqref{f,z} and \eqref{conjest_A0} yields
\begin{align}
\label{conjest_12} - z \mc{A}_0 &= 4 c \lambda^2 ( y^{ 4 \kappa } - 4 c^2 t^2 ) - 32\la c^2 \text{,} \\
\notag \frac{1}{2} \nabla^\alpha f \nabla_\alpha \mc{A}_0 &= \lambda^2 ( - 2 \kappa y^{ 6 \kappa - 1 } - 8 c^3 t^2 ) \text{.}
\end{align}
Combining \eqref{conjest_10}--\eqref{conjest_12} results in the identity
\begin{equation}
\label{conjest_20} - \mc{L} v S_{ f, z } v = - \Box_y v S_{ f, z } v + 2 \lambda ( S_{ f, z } v )^2 + \mc{B}_{ f, z } \cdot v^2 - \nabla^\alpha \left( \frac{1}{2} \mc{A}_0 \nabla_\alpha f \cdot v^2 \right) \text{,}
\end{equation}
where the coefficient $\mc{B}_{ f, z }$ is given by
\begin{align}
\label{conjest_A} \mc{B}_{ f, z } &:= \frac{1}{2} \nabla^\alpha f \nabla_\alpha \mc{A}_0 - z \mc{A}_0 \\
\notag &= \lambda^2 ( - 2 \kappa y^{ 6 \kappa - 1 } + 4 c y^{ 4 \kappa } - 24 c^3 t^2 ) - 32 \la c^2 \text{.}
\end{align}

Integrating \eqref{conjest_20} over $\mc{C}_\varepsilon$ and recalling \eqref{conjest_A} then yields
\begin{align}
\label{conjest_21} - \int_{ \mc{C}_\varepsilon } \mc{L} v S_{ f, z } v &= - \int_{ \mc{C}_\varepsilon } \Box_y v S_{ f, z } v + 2 \lambda \int_{ \mc{C}_\varepsilon } ( S_{ f, z } v )^2 \\
\notag &\qquad + \int_{ \mc{C}_\varepsilon } [ \lambda^2 ( - 2 \kappa y^{ 6 \kappa - 1 } + 4 c y^{ 4 \kappa } - 24 c^3 t^2 ) - 32 \lambda c^2 ] \cdot v^2 \\
\notag &\qquad - \frac{1}{2} \int_{ \Gamma_\varepsilon } [ \lambda^2 ( y^{ 4 \kappa } - 4 c^2 t^2 ) - 8 c \lambda ] \nabla_\nu f \cdot v^2 \text{.}
\end{align}
Notice that the bound \eqref{eqc} for $c$ implies (for all values of $n$)
\begin{equation}
\label{conjest_T1} 48 c^2 t^2 \leq 48 c^2 T^2 \leq 1 \leq y^{ 4 \kappa } \text{.}
\end{equation}
Then, with large enough $\lambda_0$ (depending on $n$ and $\kappa$), we obtain
\begin{align}
\label{conjest_22} \lambda^2 ( - 2 \kappa y^{ 6 \kappa - 1 } + 4 c y^{ 4 \kappa } - 24 c^3 t^2 ) - 32\la c^2 &\geq - 2 \kappa \lambda^2 \cdot y^{ 6 \kappa - 1 } - 32\la c^2 \\
\notag &\geq - \kappa \lambda^2 \cdot y^{ 6 \kappa - 1 } \text{.}
\end{align}
Noting in addition that
\[
| \mc{L} v S_{ f, z } v | \leq \frac{1}{ 4 \lambda } ( \mc{L} v )^2 + \lambda ( S_{ f, z } v )^2 \text{,}
\]
then~\eqref{conjest_21} and~\eqref{conjest_22} together imply
\begin{align}
\label{conjest_30} \frac{1}{ 4 \lambda } \int_{ \mc{C}_\varepsilon } ( \mc{L} v )^2 &\geq - \int_{ \mc{C}_\varepsilon } \Box_y v S_{ f, z } v + \lambda \int_{ \mc{C}_\varepsilon } ( S_{ f, z } v )^2 - \kappa \lambda^2 \int_{ \mc{C}_\varepsilon } y^{ 6 \kappa - 1 } \cdot v^2 \\
\notag &\qquad - \frac{1}{2} \int_{ \Gamma_\varepsilon } [ \lambda^2 ( y^{ 4 \kappa } - 4 c^2 t^2 ) - 8 c\la ] \nabla_\nu f \cdot v^2 \text{.}
\end{align}

At this point, the proof splits into different cases, depending on $n$.

\vspace{0.6pc}
\noindent
\emph{Case 1: $n \geq 4$.}
First, note that for large $\lambda_0$, we have
\begin{align}
\label{conjest_32} \frac{1}{9} \lambda ( S_{ f, z } v )^2 &\geq c y^{ - 4 \kappa } ( S_{ f, z } v )^2 \\
\notag &\geq c ( D_r v  )^2 + c ( 2 c t y^{ - 2 \kappa } \cdot \pd_t v + y^{ - 2 \kappa } w_{ f, z } \cdot v )^2 \\
\notag &\qquad + 2 c ( D_r v ) ( 2 c t y^{ - 2 \kappa } \cdot \pd_t v + y^{ - 2 \kappa } w_{ f, z } \cdot v ) \\
\notag &\geq \frac{1}{2} c ( D_r v  )^2 - c ( 2 c t y^{ - 2 \kappa } \cdot \pd_t v+ y^{ - 2 \kappa } w_{ f, z } \cdot v )^2 \\
\notag &\geq \frac{1}{2} c ( D_r v )^2 - 8 c^3 t^2 y^{ - 4 \kappa } \cdot ( \pd_t v)^2 - 2 c y^{ - 4 \kappa } w_{ f, z }^2 \cdot v^2 \\
\notag &\geq \frac{1}{2} c ( D_r v  )^2 - \frac{1}{6} c \cdot ( \pd_t v )^2 - 2 c y^{ - 4 \kappa } w_{ f, z }^2 \cdot v^2 \text{,}
\end{align}
where we also recalled \eqref{conjest_T1} and the definitions \eqref{f,z} and \eqref{gral_wAS} of $f$, $z$, and $S_{ f, z }$.
Moreover, recalling the formula \eqref{wAS} for $w_{ f, z }$, we obtain that
\begin{equation}
\label{conjest_33} - 18 c y^{ - 4 \kappa } w_{ f, z }^2 \cdot v^2 \geq -C( y^{-2} +r^{-2} ) \cdot v^2 \text{,}
\end{equation}
for some constant $C > 0$, depending on $n$ and $\kappa$.
Thus, for sufficiently large $\lambda_0$, it follows from \eqref{conjest_32} and \eqref{conjest_33} that
\begin{equation}
\label{conjest_34} \lambda ( S_{ f, z } v )^2 \geq \frac{9}{2} c ( D_r v )^2 - \frac{3}{2} c \cdot ( \pd_t v )^2 - C(  y^{-2} +  r^{-2} ) \cdot v^2 \text{.}
\end{equation}

Combining \eqref{conjest_30} with \eqref{conjest_34}, we obtain
\begin{align}
\label{conjest_35} \frac{1}{ 4 \lambda } \int_{ \mc{C}_\varepsilon } ( \mc{L} v )^2 &\geq - \int_{ \mc{C}_\varepsilon } \Box_y v S_{ f, z } v + \frac{9}{2} c \int_{ \mc{C}_\varepsilon } ( D_r v )^2 - \frac{3}{2} c \int_{ \mc{C}_\varepsilon } ( \pd_t v )^2 \\
\notag &\qquad - \kappa \lambda^2 \int_{ \mc{C}_\varepsilon } y^{ 6 \kappa - 1 } \cdot v^2 - C \int_{ \mc{C}_\varepsilon } ( y^{-2} + r^{-2} ) \cdot v^2 \\
\notag &\qquad - \frac{1}{2} \int_{ \Gamma_\varepsilon } [ \lambda^2 ( y^{ 4 \kappa } - 4 c^2 t^2 ) -  8 c\la ] \nabla_\nu f \cdot v^2 \text{.}
\end{align}
Applying the multiplier inequality \eqref{multineq} to \eqref{conjest_35} then results in the bound
\begin{align}
\label{conjest_40} \frac{1}{ 4 \lambda } \int_{ \mc{C}_\varepsilon } ( \mc{L} v )^2 &\geq \int_{ \mc{C}_\varepsilon } \left[ ( 1 - 4 c ) \cdot | \snabla v |^2 + \frac{1}{2} c \cdot ( \pd_tv )^2 + \frac{1}{2} c \cdot ( D_r v )^2 \right] \\
\notag &\qquad - \kappa \lambda^2 \int_{ \mc{C}_\varepsilon } y^{ 6 \kappa - 1 } \cdot v^2 - \frac{1}{2} ( n - 1 ) \kappa \int_{ \mc{C}_\varepsilon } y^{ 2 \kappa - 2 } r^{-1} \cdot v^2 \\
\notag &\qquad + \frac{1}{4} ( n - 1 ) ( n - 3 ) \int_{ \mc{C}_\varepsilon } y^{ 2 \kappa } r^{-3} \cdot v^2 \\
\notag &\qquad - C \int_{ \mc{C}_\varepsilon } ( y^{-2} + y^{ 2 \kappa - 1 } r^{-2} ) \cdot v^2 - \int_{ \Gamma_\varepsilon } S_{ f, z } v \cdot D_\nu v \\
\notag &\qquad + \frac{1}{2} \int_{ \Gamma_\varepsilon } \nabla_\nu f \cdot D_\beta v D^\beta v + \frac{1}{2} \int_{ \Gamma_\varepsilon } \nabla_\nu w_{ f, z } \cdot v^2 \\
\notag &\qquad - \frac{1}{2} \int_{ \Gamma_\varepsilon } [ \lambda^2 ( y^{ 4 \kappa } - 4 c^2 t^2 ) - 8 c\la ] \nabla_\nu f \cdot v^2 \\
\notag &\qquad + 2 \kappa ( 2 \kappa - 1 ) \int_{ \Gamma_\varepsilon } y^{ 2 \kappa - 2 } \nabla_\nu y \cdot v^2 \text{.}
\end{align}
(Here, $C$ may differ from previous lines, but still depends only on $n$ and $\kappa$.)

Let $d > 0$, and define now the (positive) quantities
\begin{align}
\label{conjest_J} J := d y^{ 2 \kappa - 2 } r^{-3} + C ( y^{-2} + y^{ 2 \kappa - 1 } r^{-2} ) \text{,} &\qquad J_0 := - \kappa \lambda^2 y^{ 6 \kappa - 1 } \text{,} \\
\notag J_1 := - \frac{1}{2} ( n - 1 ) \kappa y^{ 2 \kappa - 2 } r^{-1} \text{,} &\qquad J_2 := \frac{1}{4} ( n - 1 ) ( n - 3 ) y^{ 2 \kappa } r^{-3} \text{.}
\end{align}
Observe that for sufficiently small $d$ (depending on $n$ and $\kappa$), there is some $0 < \delta \ll 1$ (also depending on $n$ and $\kappa$) such that:
\begin{enumerate}[i)]
\item $J \leq J_2$ whenever $0 < r < \delta$.

\item $J \leq J_1$ whenever $1 - \delta < r < 1$.

\item For sufficiently large $\lambda_0$, we have that $J \leq J_0$ whenever $\delta \leq r \leq 1 - \delta$.
\end{enumerate}
Combining the above with \eqref{conjest_40} yields the desired bound \eqref{conjest}, in the case $n \geq 4$.

\vspace{0.6pc}
\noindent
\emph{Case 2: $n \leq 3$.}
For the cases $n = 1$ and $n = 3$, we first note that \eqref{eqc} implies
\begin{equation}
\label{conjest_T2} 240 c^2 t^2 \leq 240 c^2 T^2 \leq 1 \leq y^{ 4 \ka } \text{.}
\end{equation}

In this setting, we must deal with $( S_{ f, z } v )^2$ a bit differently.
To this end, we use \eqref{gral_wAS}, the fact that $\lambda_0$ is sufficiently large, and the inequality
\[
(A+B)^2 \geq (1 - 2 \ep) A^2 - \frac{1}{2\ep} (1 - 2\ep) B^2
\]
(with the values $\ep := \frac{1}{3}$, $A := y^{2 \ka} D_r v$, and $B := 2 c t (\pd_t v) + w_{ f, z} v$) in order to obtain
\begin{equation}
\label{conjest_41} \la ( S_{ f, z } v )^2 \geq 60c \left[ \frac{1}{3} y^{ 4 \ka } ( D_r v )^2 - 4 c^2 t^2 ( \pd_t v )^2 - w_{ f, z }^2 v^2 \right] \text{.}
\end{equation}
Moreover, expanding $w_{ f, z }^2$ using \eqref{wAS} and excluding terms with favorable sign yields
\begin{align}
\label{conjest_42} \la ( S_{ f, z } v )^2 &\geq 20 c y^{4\ka} ( D_r v )^2 - 240 c^3 t^2 ( \pd_t v )^2 - 540 c^3 v^2 \\
\notag &\qquad - 60 c \left[ 4 \ka^2 y^{ 4 \ka - 2 } + \frac{ (n-1)^2 }{ 4 r^2 } y^{4\ka} - \frac{ 2 \ka (n-1) }{r} y^{4 \ka - 1 } \right] v^2 \text{.}
\end{align}

The pointwise Hardy inequality \eqref{hardy}, with $q := 4 \kappa + 1$, yields
\begin{align*}
y^{ 4 \ka } ( D_r v )^2 &\geq \frac{1}{4} ( 1 - 6 \ka )^2 y^{ 4 \ka - 2 } \cdot v^2 + \frac{ (1 - 6 \ka) ( n - 1 ) }{ 2 r } y^{ 4 \ka - 1 } \cdot v^2 \\
&\qquad + \nabla^\beta \left[ \frac{ ( 1 - 6 \ka ) }{2} y^{ 4 \ka - 1 } \nabla_\beta y \cdot v^2 \right] \text{.}
\end{align*}
Combining the above with \eqref{conjest_T2} and \eqref{conjest_42}, and noting that
\[
\frac{15}{4} ( 1 - 6 \kappa )^2 > 240 \kappa^2 \text{,}
\]
we then obtain the bound
\begin{align}
\label{conjest_43} \la ( S_{ f, z } v )^2 &\geq 5c ( D_r v )^2 - c ( \pd_t v )^2 - 15 c (n-1)^2 y^{4\ka} r^{-2} v^2 \\
\notag &\qquad - C (n-1) y^{4\ka-1} r^{-1} v^2 + \nabla^\beta \left[ \frac{ 15 c ( 1 - 6 \kappa ) }{2} y^{ 4\ka - 1 } \nabla_\beta y \cdot v^2 \right] \text{,}
\end{align}
where $C > 0$ depends on $n$ and $\kappa$.

Now, applying the multiplier inequality \eqref{multineq} and \eqref{conjest_43} to \eqref{conjest_30}, we see that
\begin{align}
\label{conjest_50} \frac{1}{ 4 \lambda } \int_{ \mc{C}_\varepsilon } ( \mc{L} v )^2 &\geq \int_{ \mc{C}_\varepsilon } \left[ ( 1 - 4 c ) | \snabla v |^2 + c ( \pd_tv )^2 + c ( D_r v )^2 \right] \\
\notag &\qquad - \kappa \lambda^2 \int_{ \mc{C}_\varepsilon } y^{ 6 \kappa - 1 } \cdot v^2 - \frac{1}{2} ( n - 1 ) \kappa \int_{ \mc{C}_\varepsilon } y^{ 2 \kappa - 2 } r^{-1} \cdot v^2 \\
\notag &\qquad + \frac{1}{2} ( n - 1 ) ( n - 4 ) \kappa \int_{ \mc{C}_\varepsilon } y^{ 2 \kappa - 1 } r^{-2} \cdot v^2 \\
\notag &\qquad - 15 c (n-1)^2 \int_{ \mc{C}_\varepsilon } y^{ 4 \ka } r^{-2} \cdot v^2 \\
\notag &\qquad - C ( n - 1 ) \int_{ \mc{C}_\varepsilon } y^{ 4 \ka - 1 } r^{-1} \cdot v^2 \\
\notag &\qquad - \int_{ \Gamma_\varepsilon } S_{ f, z } v \cdot D_\nu v + \frac{1}{2} \int_{ \Gamma_\varepsilon } \nabla_\nu f \cdot D_\beta v D^\beta v \\
\notag &\qquad + \frac{1}{2} \int_{ \Gamma_\varepsilon } \nabla_\nu w_{ f, z } \cdot v^2 + 2 \kappa ( 2 \kappa - 1 ) \int_{ \Gamma_\varepsilon } y^{ 2 \kappa - 2 } \nabla_\nu y \cdot v^2 \\
\notag &\qquad - \frac{1}{2} \int_{ \Gamma_\varepsilon } [ \lambda^2 ( y^{ 4 \kappa } - 4 c^2 t^2 ) - 8 c \la ] \nabla_\nu f \cdot v^2 \\
\notag &\qquad + c_2 \int_{ \Gamma_\varepsilon } y^{ 4\ka - 1 } \nabla_\nu y \cdot v^2 \text{.}
\end{align}
For $n = 1$, the bound \eqref{conjest_50} immediately implies \eqref{conjest}.

For the remaining case $n = 3$, we also note from \eqref{eqc} that
\begin{equation}
\label{conjest_51} \frac{1}{2} ( n - 1 ) ( n - 4 ) \kappa y^{ 2 \kappa - 1 } r^{-2} - 15 c ( n - 1 )^2 y^{ 4 \kappa } r^{-2} \geq - \frac{1}{2} \kappa y^{ 2 \kappa - 1 } r^{-2} \text{.}
\end{equation}
To control the remaining bulk integrand $- C ( n - 1 ) y^{ 4 \kappa - 1 } r^{-1} \cdot v^{-2}$, we define
\begin{align}
\label{conjest_K} K: = d y^{ 2 \kappa - 2 } r^{-2} + C ( n - 1 ) y^{ 4 \kappa - 1 } r^{-1} \text{,} &\qquad K_0: = - \kappa \lambda^2 y^{ 6 \kappa - 1 } \text{,} \\
\notag K_1: = - \frac{1}{2} ( n - 1 ) \kappa y^{ 2 \kappa - 2 } r^{-1} \text{,} &\qquad K_2: = - \frac{1}{2} \kappa y^{ 2 \kappa - 1 } r^{-2} \text{.}
\end{align}

Like for the $n \geq 4$ case, as long as $d$ is sufficiently small (depending on $n$ and $\kappa$), then there exists $0 < \delta \ll 1$ (depending on $n$ and $\kappa$) such that:
\begin{enumerate}[i)]
\item $K \leq K_2$ whenever $0 < r < \delta$.

\item $K \leq K_1$ whenever $1 - \delta < r < 1$.

\item For large enough $\lambda_0$, we have that $K \leq K_0$ whenever $\delta \leq r \leq 1 - \delta$.
\end{enumerate}
Combining the above with \eqref{conjest_50} and \eqref{conjest_51} yields \eqref{conjest} for $n = 3$.
\end{proof}

\subsection{Boundary Limits}

In this subsection, we derive and control the limits of the boundary terms in \eqref{conjest} when $\varepsilon \searrow 0$.
More specifically, we show the following:

\begin{lemma} \label{L.bdrylim}
Let $\Gamma_\varepsilon^\pm$ be as in \eqref{Gamma_eps}.
Then, for $\lambda \geq \lambda_0$,
\begin{align}
\label{bdrylim_outer} - c_3 \int_\Gamma e^{ 2 \lambda f } ( \mc{N}_\kappa u )^2 &\leq \liminf_{ \varepsilon \searrow 0 } \left[ \int_{ \Gamma_\varepsilon^+ } \nabla_\nu f \cdot D_\beta v D^\beta v - 2 \int_{ \Gamma_\varepsilon^+ } S_{ f, z } v D_\nu v \right] \\
\notag &\qquad - \lim_{ \varepsilon \searrow 0 } \int_{ \Gamma_\varepsilon^+ } [ \lambda^2 ( y^{ 4 \kappa } - 4 c^2 t^2 ) - 8 c\la ] \nabla_\nu f \cdot v^2 \\
\notag &\qquad + \lim_{ \varepsilon \searrow 0 } \int_{ \Gamma_\varepsilon^+ } \nabla_\nu w_{ f, z } \cdot v^2 \\
\notag &\qquad + 4 \kappa ( 2 \kappa - 1 ) \lim_{ \varepsilon \searrow 0 } \int_{ \Gamma_\varepsilon^+ } y^{ 2 \kappa - 2 } \nabla_\nu y \cdot v^2 \text{,} \\
\notag 0 &= \lim_{ \varepsilon \searrow 0 } \int_{ \Gamma_\varepsilon^+ } y^{ 4 \kappa - 1 } \nabla_\nu y \cdot v^2 \text{,}
\end{align}
where the constant $c_3 > 0$ depends on $\kappa$.
In addition, for $\lambda \geq \lambda_0$,
\begin{align}
\label{bdrylim_inner} 0 &\leq \lim_{ \varepsilon \searrow 0 } \left[ \int_{ \Gamma_\varepsilon^- } \nabla_\nu f \cdot D_\beta v D^\beta v - 2 \int_{ \Gamma_\varepsilon^- } S_{ f, z } v D_\nu v \right] \\
\notag &\qquad - \lim_{ \varepsilon \searrow 0 } \int_{ \Gamma_\varepsilon^- } [ \lambda^2 ( y^{ 4 \kappa } - 4 c^2 t^2 ) - 8 c\la ] \nabla_\nu f \cdot v^2 \\
\notag &\qquad + \lim_{ \varepsilon \searrow 0 } \int_{ \Gamma_\varepsilon^- } \nabla_\nu w_{ f, z } \cdot v^2 + 4 \kappa ( 2 \kappa - 1 ) \lim_{ \varepsilon \searrow 0 } \int_{ \Gamma_\varepsilon^- } y^{ 2 \kappa - 2 } \nabla_\nu y \cdot v^2 \text{,} \\
\notag 0 &\leq \lim_{ \varepsilon \searrow 0 } \int_{ \Gamma_\varepsilon^- } y^{ 4 \kappa - 1 } \nabla_\nu y \cdot v^2 \text{.}
\end{align}
\end{lemma}

\begin{proof}
First, note that on $\Gamma_\varepsilon^\pm$, we have
\begin{equation}
\label{bdrylim_0} \nu |_{ \Gamma_\varepsilon^\pm } = \pm \partial_r \text{,} \qquad \nabla_\nu y |_{ \Gamma_\varepsilon^\pm } = \mp 1 \text{,} \qquad \nabla_\nu f |_{ \Gamma_\varepsilon^\pm } = \pm y^{ 2 \kappa } |_{ \Gamma_\varepsilon^\pm } \text{.}
\end{equation}
Moreover, note that \eqref{f,z} and \eqref{gral_wAS} imply
\begin{equation}
\label{bdrylim_1} S_{ f, z } v = y^{ 2 \kappa } D_r v + 2 c t \cdot \partial_t v + w_{ f, z } \cdot v \text{.}
\end{equation}

We begin with the outer limits \eqref{bdrylim_outer}.
The main observation is that by \eqref{f,z} and by the assumption that $u$ is boundary admissible (see Definition \ref{admissible}), we have
\begin{align}
\label{bdrylim_oa} \lim_{ \varepsilon \searrow 0 } \int_{ \Gamma_\varepsilon^+ } y^{ 2 \kappa } ( \partial_t v )^2 &= 0 \text{,} \\
\notag \lim_{ \varepsilon \searrow 0 } \int_{ \Gamma_\varepsilon^+ } y^{ 2 \kappa } ( D_r v )^2 &= \int_\Gamma e^{ 2 \lambda f } ( \mc{N}_\kappa u )^2 \text{,} \\
\notag \lim_{ \varepsilon \searrow 0 } \int_{ \Gamma_\varepsilon^+ } y^{ - 2 + 2 \kappa } v^2 &= ( 1 - 2 \kappa )^{-2} \int_\Gamma e^{ 2 \lambda f } ( \mc{N}_\kappa u )^2 \text{.}
\end{align}
We also recall that we have assumed $- \frac{1}{2} < \kappa < 0$.

For the first boundary term, we apply \eqref{bdrylim_0} and \eqref{bdrylim_oa} to obtain
\begin{align}
\label{bdrylim_o1} \liminf_{ \varepsilon \searrow 0 } \int_{ \Gamma_\varepsilon^+ } \nabla_\nu f \cdot D_\beta v D^\beta v &\geq \lim_{ \varepsilon \searrow 0 } \int_{ \Gamma_\varepsilon^+ } y^{ 2 \kappa } [ - ( \partial_t v )^2 + ( D_r v )^2 ] \\
\notag &= \int_\Gamma e^{ 2 \lambda f } ( \mc{N}_\kappa u )^2 \text{.} 
\end{align}
Next, expanding $S_{ f, z } v$ using \eqref{bdrylim_1}, noting from \eqref{wAS} that the leading-order behavior of $w_{ f, z }$ near $\Gamma$ is $- 2 \kappa \cdot y^{ 2 \kappa - 1 }$, and applying \eqref{bdrylim_oa}, we obtain that
\begin{align}
\label{bdrylim_o2} - 2 \lim_{ \varepsilon \searrow 0 } \int_{ \Gamma_\varepsilon^+ } S_{ f, z } v D_\nu v &= - 2 \lim_{ \varepsilon \searrow 0 } \int_{ \Gamma_\varepsilon^+ } [ y^{ 2 \kappa } ( D_r v )^2 + 2 c t \partial_t v D_r v + w_{ f, z } v D_r v ] \\
\notag &= - 2 \int_\Gamma e^{ 2 \lambda f } ( \mc{N}_\kappa u )^2 + 4 \kappa \lim_{ \varepsilon \searrow 0 } \int_{ \Gamma_\varepsilon^+ } y^{ 2 \kappa - 1 } v D_r v \\
\notag &= \left( - 2 + \frac{ 4 \kappa }{ 1 - 2 \kappa } \right) \int_\Gamma e^{ 2 \lambda f } ( \mc{N}_\kappa u )^2 \text{.}
\end{align}

The remaining outer boundary terms are treated similarly.
By \eqref{bdrylim_0} and \eqref{bdrylim_oa},
\begin{align}
\label{bdrylim_o3} - \lim_{ \varepsilon \searrow 0 } \int_{ \Gamma_\varepsilon^+ } [ \lambda^2 ( y^{ 4 \kappa } - 4 c^2 t^2 ) - 8 c\la ] \nabla_\nu f \cdot v^2 &= - \lim_{ \varepsilon \searrow 0 } \int_{ \Gamma_\varepsilon^+ } y^{ 6 \kappa } v^2 = 0 \text{,} \\
\notag 4 \kappa ( 2 \kappa - 1 ) \lim_{ \varepsilon \searrow 0 } \int_{ \Gamma_\varepsilon^+ } y^{ 2 \kappa - 2 } \nabla_\nu y \cdot v^2 &= \frac{ 4 \kappa }{ 1 - 2 \kappa } \int_\Gamma e^{ 2 \lambda f } ( \mc{N}_\kappa u )^2 \text{.}
\end{align}
Moreover, by \eqref{wAS} and \eqref{bdrylim_0}, we see that the leading-order behavior of $\partial_r w_{ f, z }$ is given by $- 2 \kappa ( 1 - 2 \kappa ) y^{ 2 \kappa - 2 }$.
Combining this with \eqref{bdrylim_0} and \eqref{bdrylim_oa} yields
\begin{align}
\label{bdrylim_o4} \lim_{ \varepsilon \searrow 0 } \int_{ \Gamma_\varepsilon^+ } \nabla_\nu w_{ f, z } \cdot v^2 &= - 2 \kappa ( 1 - 2 \kappa ) \lim_{ \varepsilon \searrow 0 } \int_\Gamma y^{ 2 \kappa - 2 } v^2 \\
\notag &= - \frac{ 2 \kappa }{ 1 - 2 \kappa } \int_\Gamma e^{ 2 \lambda f } ( \mc{N}_\kappa u )^2 \text{.}
\end{align}

Summing \eqref{bdrylim_o1}--\eqref{bdrylim_o4} yields the first part of \eqref{bdrylim_outer}.
The second part of \eqref{bdrylim_outer} similarly follows by applying \eqref{bdrylim_0} and \eqref{bdrylim_oa}.

Next, for the interior limits \eqref{bdrylim_inner}, we split into two cases:

\vspace{0.6pc}
\noindent
\emph{Case 1: $n \geq 3$.}
In this case, we begin by noting that the volume of $\Gamma_\varepsilon^-$ satisfies
\begin{equation}
\label{bdrylim_ia} | \Gamma_\varepsilon^- | \lesssim_{ T, n } \varepsilon^{ n - 1 } \text{.}
\end{equation}
Furthermore, since $u$ is smooth on $\mc{C}$, then \eqref{f,z} and \eqref{eqv} imply that $\partial_t v$, $\snabla v$, $D_r v$, and $v$ are all uniformly bounded whenever $r$ is sufficiently small.
Combining the above with \eqref{wAS}, \eqref{bdrylim_0}, \eqref{bdrylim_1}, we obtain that the following limits vanish:
\begin{align}
\label{bdrylim_i1} 0 &= \lim_{ \varepsilon \searrow 0 } \left[ \int_{ \Gamma_\varepsilon^- } \nabla_\nu f \cdot D_\beta v D^\beta v - 2 \int_{ \Gamma_\varepsilon^- } S_{ f, z } v D_\nu v \right] \\
\notag &\qquad - \lim_{ \varepsilon \searrow 0 } \int_{ \Gamma_\varepsilon^- } [ \lambda^2 ( y^{ 4 \kappa } - 4 c^2 t^2 ) - 8 c\la ] \nabla_\nu f \cdot v^2 \\
\notag &\qquad + 4 \kappa ( 2 \kappa - 1 ) \lim_{ \varepsilon \searrow 0 } \int_{ \Gamma_\varepsilon^- } y^{ 2 \kappa - 2 } \nabla_\nu y \cdot v^2 \text{,} \\
\notag 0 &= \lim_{ \varepsilon \searrow 0 } \int_{ \Gamma_\varepsilon^- } y^{ 4 \kappa - 1 } \nabla_\nu y \cdot v^2 \text{.}
\end{align}

This leaves only one remaining limit in \eqref{bdrylim_inner}; for this, we note, from \eqref{wAS}, that the leading-order behavior of $- \partial_r w_{ f, z }$ near $r = 0$ is $\frac{1}{2} ( n - 1 ) r^{-2} y^{ 2 \kappa }$.
As a result,
\begin{align}
\label{bdrylim_i2} \lim_{ \varepsilon \searrow 0 } \int_{ \Gamma_\varepsilon^- } \nabla_\nu w_{ f, z } \cdot v^2 &= \frac{ n - 1 }{2} \lim_{ \varepsilon \searrow 0 } \int_{ \Gamma_\varepsilon^- } r^{-2} y^{ 2 \kappa } v^2 \\
\notag &= \begin{cases} 0 & \quad n > 3 \\ C \int_{ -T }^T | v ( t, 0 ) |^2 dt & \quad n = 3 \end{cases} \text{,}
\end{align}
where the last integral is over the line $r = 0$, and where the constant $C$ depends only on $n$.
Combining \eqref{bdrylim_i1} and \eqref{bdrylim_i2} yields \eqref{bdrylim_inner} in this case.

\vspace{0.6pc}
\noindent
\emph{Case 2: $n = 1$.}
Here, we can no longer rely on \eqref{bdrylim_ia} to force most limits to vanish, so we must examine all the terms more carefully.

First, from \eqref{wAS}, \eqref{bdrylim_0}, \eqref{bdrylim_1}, we have that
\begin{align*}
&\int_{ \Gamma_\varepsilon^- } \nabla_\nu f \cdot D_\beta v D^\beta v - 2 \int_{ \Gamma_\varepsilon^- } S_{ f, z } v D_\nu v \\
\notag &\quad = \int_{ \Gamma_\varepsilon^- } y^{ 2 \kappa } [ ( \pd_t v )^2 + ( D_r v )^2 ] + \int_{ \Gamma_\varepsilon^- } [ 4 c t \cdot \partial_t v D_r v - 4 \kappa y^{ 2 \kappa - 1 } v D_r v ] \text{.}
\end{align*}
Recalling also our assumption \eqref{eqc} for $c$, we conclude from the above that
\begin{align}
\label{bdrylim_a10} \lim_{ \varepsilon \searrow 0 } \left[ \int_{ \Gamma_\varepsilon^- } \nabla_\nu f \cdot D_\beta v D^\beta v - 2 \int_{ \Gamma_\varepsilon^- } S_{ f, z } v D_\nu v \right] &\geq - C \lim_{ \varepsilon \searrow 0 } \int_{ \Gamma_\varepsilon^- } y^{ 2 \kappa - 2 } v^2 \\
\notag &= - C \int_{-T}^T | v ( t, 0 ) |^2 dt \text{,}
\end{align}
where the last integral is over the line $r = 0$, and where $C$ depends only on $\kappa$.
Moreover, letting $\lambda_0$ be sufficiently large and recalling \eqref{eqc} and \eqref{bdrylim_0}, we obtain
\begin{equation}
\label{bdrylim_a11} - \lim_{ \varepsilon \searrow 0 } \int_{ \Gamma_\varepsilon^- } [ \lambda^2 ( y^{ 4 \kappa } - 4 c^2 t^2 ) - 8 c\la ] \nabla_\nu f \cdot v^2 \geq \tilde{C} \lambda^2 \int_{-T}^T | v ( t, 0 ) |^2 dt \text{,}
\end{equation}
for some constant $\tilde{C} > 0$.

Next, applying \eqref{wAS} and \eqref{bdrylim_0} in a similar manner as before, we obtain inequalities for the remaining limits in the right-hand side of \eqref{bdrylim_inner}:
\begin{align}
\label{bdrylim_a12} \lim_{ \varepsilon \searrow 0 } \int_{ \Gamma_\varepsilon^- } \nabla_\nu w_{ f, z } \cdot v^2 &\geq - C \int_{-T}^T | v ( t, 0 ) |^2 dt \text{,} \\
\notag 4 \kappa ( 2 \kappa - 1 ) \lim_{ \varepsilon \searrow 0 } \int_{ \Gamma_\varepsilon^- } y^{ 2 \kappa - 2 } \nabla_\nu y \cdot v^2 &\geq - C \int_{-T}^T | v ( t, 0 ) |^2 dt \text{,} \\
\notag \lim_{ \varepsilon \searrow 0 } \int_{ \Gamma_\varepsilon^- } y^{ 4 \kappa - 1 } \nabla_\nu y \cdot v^2 &= 2 \int_{-T}^T | v ( t, 0 ) |^2 dt \text{.}
\end{align}
Here, $C$ denotes various positive constants that depend on $\kappa$.
Finally, combining \eqref{bdrylim_a10}--\eqref{bdrylim_a12} and taking $\lambda_0$ to be sufficiently large results in \eqref{bdrylim_inner}.
\end{proof}

\subsection{Completion of the Proof}

We are now in position to complete the proof of Theorem \ref{T.Carleman}.
First, recalling the definitions \eqref{f,z} and \eqref{eqv} of $f$ and $v$ and the fact that $c^2 t^2 \lesssim 1$ by our assumption \eqref{eqc}, we have that
\begin{align}
\label{conjest_58} e^{ 2 \la f } ( \partial_t u )^2 &\lesssim ( \pd_t v )^2 + \la^2 c^2 t^2 v^2 \lesssim ( \partial_t v )^2 + \la^2 y^{ 6 \ka - 1 } v^2 \text{,} \\
\notag e^{ 2 \la f } ( D_r u )^2 &\lesssim ( D_r v )^2 + \la^2 y^{ 4 \ka } v^2 \lesssim ( D_r v )^2 + \la^2 y^{ 6 \ka - 1 } v^2 \text{,} \\
\notag e^{ 2 \la f } | \snabla u |^2 &= |\snabla v|^2 \text{.}
\end{align}
Furthermore, by \eqref{Box_y_kappa} and \eqref{eqv}, we observe that
\begin{equation}
\label{conjest_59} ( \mc{L} v )^2 \leq 2 e^{ 2 \la f } [ (\Box_\kappa u )^2 + \kappa ( n - 1 ) y^{-2} r^{-2} \cdot u^2 ] \text{.}
\end{equation}

Therefore, using these bounds in Lemma~\ref{L.conjest}, it follows that
\begin{align}
\label{conjest_60} &2 \int_{ \mc{C}_\varepsilon } e^{ 2 \la f } ( \Box_\kappa u )^2 + 2 \kappa ( n - 1 ) \int_{ \mc{C}_\varepsilon } e^{ 2 \la f } y^{-1} r^{-1} \cdot u^2 \\
\notag &\quad \geq C \la \int_{ \mc{C}_\varepsilon } e^{ 2 \la f } [ ( \pd_t u )^2 + | \snabla u |^2 + ( D_r u )^2 ] + C \lambda^3 \int_{ \mc{C}_\varepsilon } e^{ 2 \lambda f } y^{ 6 \kappa - 1 } u^2 \\
\notag &\quad\qquad + 2 \la \int_{ \Gamma_\varepsilon } \nabla_\nu f \cdot D_\beta v D^\beta v - 4 \la \int_{ \Gamma_\varepsilon } S_{ f, z } v \cdot D_\nu v \\
\notag &\quad\qquad - 2 \la \int_{ \Gamma_\varepsilon } [ \lambda^2 ( y^{ 4 \kappa } - 4 c^2 t^2 ) - 8 c\la ] \nabla_\nu f \cdot v^2 \\
\notag &\quad\qquad + 2 \la \int_{ \Gamma_\varepsilon } \nabla_\nu w_{ f, z } \cdot v^2 + 8\la \kappa ( 2 \kappa - 1 ) \int_{ \Gamma_\varepsilon } y^{ 2 \kappa - 2 } \nabla_\nu y \cdot v^2 \\
\notag &\quad\qquad + \begin{cases}
  C \lambda \int_{ \mc{C}_\varepsilon } e^{ 2 \la f } y^{ 2 \kappa - 2 } r^{-3} \cdot u^2 & \quad n \geq 4 \\
  C \lambda \int_{ \mc{C}_\varepsilon } e^{ 2 \la f } y^{ 2 \kappa - 2 } r^{-2} \cdot u^2 + 4 c_2 \lambda \int_{ \Gamma_\varepsilon } y^{ 4 \kappa - 1 } \nabla_\nu y \cdot v^2 & \quad n = 3 \\
  4 c_2 \lambda \int_{ \Gamma_\varepsilon } y^{ 4 \kappa - 1 } \nabla_\nu y \cdot v^2 & \quad n = 1
\end{cases} \,,
\end{align}
for some constant $C > 0$ depending on $n$ and $\kappa$.
Note that if $\lambda_0$ is sufficiently large, then the last term on the left-hand side of \eqref{conjest_60} can be absorbed into the last term on the right-hand side of \eqref{conjest_60} (for all values of $n$).
From this, we obtain
\begin{align}
\label{conjest_70} \int_{ \mc{C}_\varepsilon } e^{ 2 \la f } ( \Box_\kappa u )^2 &\geq C \la \int_{ \mc{C}_\varepsilon } e^{ 2 \la f } [ ( \pd_t u )^2 + | \snabla u |^2 + ( D_r u )^2 + \lambda^2 y^{ 6 \kappa - 1 } u^2 ] \\
\notag &\qquad + \begin{cases}
  C \lambda \int_{ \mc{C}_\varepsilon } e^{ 2 \la f } y^{ 2 \kappa - 2 } r^{-3} \cdot u^2 & \quad n \geq 4 \\
  C \lambda \int_{ \mc{C}_\varepsilon } e^{ 2 \la f } y^{ 2 \kappa - 2 } r^{-2} \cdot u^2 & \quad n = 3 \\
  0 & \quad n = 1
\end{cases} \\
\notag &\qquad + \la \int_{ \Gamma_\varepsilon } \nabla_\nu f \cdot D_\beta v D^\beta v - 2 \la \int_{ \Gamma_\varepsilon } S_{ f, z } v \cdot D_\nu v \\
\notag &\qquad - \la \int_{ \Gamma_\varepsilon } [ \lambda^2 ( y^{ 4 \kappa } - 4 c^2 t^2 ) - 8 c\la ] \nabla_\nu f \cdot v^2 \\
\notag &\qquad + \la \int_{ \Gamma_\varepsilon } \nabla_\nu w_{ f, z } \cdot v^2 + 4 \la \kappa ( 2 \kappa - 1 ) \int_{ \Gamma_\varepsilon } y^{ 2 \kappa - 2 } \nabla_\nu y \cdot v^2 \\
\notag &\qquad + \begin{cases}
  0 & \quad n \geq 4 \\
  2 c_2 \lambda \int_{ \Gamma_\varepsilon } y^{ 4 \kappa - 1 } \nabla_\nu y \cdot v^2 & \quad n \leq 3
\end{cases} \text{.}
\end{align}

Finally, the desired inequality \eqref{Carleman} follows by taking the limit $\varepsilon \searrow 0$ in \eqref{conjest_70} and applying all the inequalities from Lemma \ref{L.bdrylim}.

\section{Observability} \label{S.Observability}

Our aim in this section is to show that the Carleman estimates of Theorem \ref{T.Carleman} imply a boundary observability property for solutions to wave equations on the cylindrical spacetime $\mc{C}$ containing potentials that are critically singular at the boundary $\Gamma$.
More specifically, we establish the following result, which is a precise and a slightly stronger version of the result stated in Theorem \ref{T.Observability0}.

\begin{theorem}\label{T.Observability}
Assume $n \neq 2$, and fix $-\frac{1}{2} < \ka < 0$.
Let $u$ be a solution to
\begin{equation}
\label{weqn} \Box_\ka u = D_X u + V u \text{,}
\end{equation}
on $\bar{\mc{C}}$, where the vector field $X: \mc{C} \rightarrow \R^{1+n}$ and the potential $V: \mc{C} \rightarrow \R$ satisfy
\begin{equation}
\label{hypXV} |X| \lesssim 1\text{,} \qquad | V | \lesssim \frac{1}{y} + \frac{ n - 1 }{r} \text{,}
\end{equation}
In addition, assume that:
\begin{enumerate}[i)]
\item $u$ is boundary admissible (in the sense of Definition \ref{admissible}).

\item $u$ has finite twisted $H^1$-energy for any $\tau \in ( -T, T )$:
\begin{equation}
\label{H1} E_1 [u] ( \tau ) = \int_{ \mc{C} \cap \{ t = \tau \} } ( ( \pd_t u )^2 + ( D_r u )^2 + | \snabla u |^2 + u^2 ) < \infty \text{.}
\end{equation}
\end{enumerate}
Then, for sufficiently large observation time $T$ satisfying
\begin{equation} \label{obstime}
\begin{cases}
  T > \frac{ 4 \sqrt{3} }{ 1 + 2 \ka } & \quad n \geq 4 \\
  T > \max \left\{ \frac{ 4 \sqrt{15} }{ 1 + 2 \ka }, \frac{ 2 \sqrt{30} }{ \sqrt{ | \ka | ( 1 + 2 \ka ) } } \right\} & \quad n = 3 \\
  T > \frac{ 4 \sqrt{15} }{ 1 + 2 \ka } & \quad n = 1
\end{cases} \text{,}
\end{equation}
we have the boundary observability inequality
\begin{equation}
\label{Observability} \int_\Gamma (\mc{N}_\ka u)^2 \gtrsim E_1 [u] (0) \text{,}
\end{equation}
where the constant of the inequality depends on $n$, $\kappa$, $T$, $X$, and $V$.
\end{theorem}


\subsection{Preliminary Estimates}

In order to prove Theorem \ref{T.Observability}, we require preliminary estimates.
The first is a Hardy estimate to control singular integrands:

\begin{lemma} \label{L.hardy}
Assume the hypotheses of Theorem \ref{T.Observability}.
Then,
\begin{equation}
\label{hardy_int} \int_{ \mc{C} \cap \{ t_0 < t < t_1 \} } \left( \frac{1}{ y^2 } + \frac{ n - 1 }{ r^2 } \right) u^2 \lesssim \int_{ \mc{C} \cap \{ t_0 < t < t_1 \} } (D_r u )^2 \text{,}
\end{equation}
for any $-T \leq t_0 < t_1 \leq T$, where the constant depends only on $n$ and $\kappa$.
\end{lemma}

\begin{proof}
The inequality \eqref{hardy}, with $q=1$, yields
\[
( D_r u )^2 \geq \frac{1}{8} ( 1 - 2 \ka )^2 \frac{ u^2 }{ y^2 } + \frac{ (n-1) }{9} \frac{ u^2 }{ r^2 } + \frac{ ( 1 - 2 \ka ) }{2} \nab^\be ( \nab_\be y \cdot y^{-1} u^2 ) \text{.}
\]
Letting $0 < \varepsilon \ll 1$ and integrating the above over $\mc{C} \cap \{ t_0 < t < t_1 \}$ yields
\begin{align*}
\int_{ \mc{C}_\varepsilon \cap \{ t_0 < t < t_1 \} } ( D_r u )^2 &\geq C \int_{ \mc{C}_\varepsilon \cap \{ t_0 < t < t_1 \} } \left( \frac{1}{ y^2 } + \frac{ n - 1 }{ r^2 } \right) u^2 \\
&\qquad - \frac{ ( 1 - 2 \ka ) }{2} \int_{ \Gamma_\varepsilon^+ \cap \{ t_0 < t < t_1 \} } y^{-1} u^2 \\
&\qquad + \frac{ ( 1 - 2 \ka ) }{2} \int_{ \Gamma_\varepsilon^- \cap \{ t_0 < t < t_1 \} } y^{-1} u^2 \\
&\geq C \int_{ \mc{C}_\varepsilon \cap \{ t_0 < t < t_1 \} } \left( \frac{1}{ y^2 } + \frac{ n - 1 }{ r^2 } \right) u^2 \\
&\qquad - \frac{ ( 1 - 2 \ka ) }{2} \int_{ \Gamma_\varepsilon^+ \cap \{ t_0 < t < t_1 \} } y^{-1} u^2 \text{.}
\end{align*}
(Here, we have also made use of the identities \eqref{bdrylim_0}.)
Letting $\varepsilon \searrow 0$ and recalling that $u$ is boundary admissible results in the estimate \eqref{hardy_int}.
\end{proof}

We will also need the following energy estimate for solutions to \eqref{weqn}: 

\begin{lemma} \label{L.energy}
Assume the hypotheses of Theorem \ref{T.Observability}.
Then,
\begin{equation}
\label{energyineq}
E_1 [u] ( t_1 ) \leq e^{ M | t_1 - t_0 |} E_1 [u] ( t_0 ) \text{,} \qquad t_0, t_1 \in ( -T, T ) \text{,}
\end{equation}
where the constant $M$ depends on $n$, $\kappa$, $X$, and $V$.
\end{lemma}

\begin{proof}
We assume for convenience that $t_0 < t_1$; the opposite case can be proved analogously.
By a standard density argument, we can assume $u$ is smooth within $\mc{C}$.
Fix now a sufficiently small $0 < \varepsilon \ll 1$, and define
\begin{equation}
\label{energy_0} E_{ 1, \varepsilon } [u] ( \tau ) = \int_{ \mc{C}_\varepsilon \cap \{ t = \tau \} } ( ( \pd_t u )^2 + ( D_r u )^2 + | \snabla u |^2 + u^2 ) \text{.}
\end{equation}

Differentiating $E_{ 1, \varepsilon } [u]$ and integrating by parts, we obtain, for any $\tau \in ( -T, T )$,
\begin{align}
\label{energy_1} \frac{d}{ d \tau } E_{ 1, \varepsilon } [u] (\tau) &= 2 \int_{ \mc{C}_\varepsilon \cap \{ t = \tau \} } ( \pd_{tt} u \pd_tu + D^j u D_j \pd_t u + u \pd_t u ) \\
\notag &= - 2 \int_{ \mc{C}_\varepsilon \cap \{ t = \tau \} } \pd_t u ( \Box_y u - u ) + 2 \int_{ \Ga_\varepsilon \cap \{ t = \tau \} } \partial_t u D_\nu u \text{.}
\end{align}
Note that \eqref{Box_y_kappa}, \eqref{weqn}, and \eqref{hypXV} imply
\begin{align*}
| \Box_y u | &\lesssim \left| D_X u + V u + \frac{ ( n - 1 ) \kappa }{ r y } u \right| \\
&\lesssim | \pd_t u | + | \snabla u | + | D_r u | + \left( \frac{1}{y} + \frac{n-1}{r} \right) |u| \text{.}
\end{align*}
Combining the above with \eqref{energy_1} yields
\begin{align*}
\frac{d}{ d \tau } E_{ 1, \varepsilon } [u] (\tau) &\leq C \cdot E_1 [u] (\tau) + C \cdot E_1^{ \frac{1}{2} } [u] (\tau) \left[ \int_{ \mc{C} \cap \{ t = \tau \} } \left( \frac{1}{ y^2 } + \frac{ n - 1 }{ r^2 } \right) u^2 \right]^{ \frac{1}{2} } \\
\notag &\qquad + 2 \int_{ \Ga_\varepsilon \cap \{ t = \tau \} } \partial_t u D_\nu u \text{.}
\end{align*}

Next, integrating the above in $\tau$ and applying Lemma \ref{L.hardy}, we obtain
\begin{equation}
\label{energy_2} E_{ 1, \varepsilon } [u] ( t_1 ) \leq E_1 [u] ( t_0 ) + C \int_{ t_0 }^{ t_1 } E_1 [u] (\tau) \, d \tau + 2 \int_{ \Ga_\varepsilon \cap \{ t_0 < t < t_1 \} } \partial_t u D_\nu u \text{.}
\end{equation}
Since $u$ is boundary admissible, it follows that
\begin{equation}
\label{energy_3} \lim_{ \varepsilon \searrow 0 } \int_{ \Ga_\varepsilon^+ \cap \{ t_0 < t < t_1 \} } \partial_t u D_\nu u = 0 \text{.}
\end{equation}
Moreover, since $\nu$ points radially along $\Gamma_\varepsilon^-$, then by symmetry,
\begin{equation}
\label{energy_4} \lim_{ \varepsilon \searrow 0 } \int_{ \Ga_\varepsilon^- \cap \{ t_0 < t < t_1 \} } \partial_t u D_\nu u = 0 \text{.}
\end{equation}
(Alternatively, when $n > 1$, we can also use \eqref{bdrylim_ia}.)

Letting $\varepsilon \searrow 0$ in \eqref{energy_2} and applying \eqref{energy_3}--\eqref{energy_4}, we conclude that
\[
E_1 [u] ( t_1 ) \leq E_1 [u] ( t_0 ) + C \int_{ t_0 }^{ t_1 } E_1 [u] (\tau) \, d \tau \text{.}
\]
The estimate \eqref{energyineq} now follows from the Gr\"onwall inequality.
\end{proof}

\subsection{Proof of Theorem \ref{T.Observability}}

Assume the hypotheses of Theorem \ref{T.Observability}, and set
\begin{equation}
\label{estobs_c} c = \begin{cases}
  \frac{ 1 }{ 4 \sqrt{3} \cdot T} & \quad n \geq 4 \\
  \min \left\{ \frac{1}{4 \sqrt{15} \cdot T}, \frac{|\ka|}{120} \right\} & \quad n = 3 \\
  \frac{ 1}{ 4 \sqrt{15} \cdot T} & \quad n = 1
\end{cases} \text{.}
\end{equation}
Note, in particular, that \eqref{estobs_c} and \eqref{obstime} imply that the conditions \eqref{eqc} hold.

Moreover, we define the function $f$ as in the statement of Theorem \ref{T.Carleman}, with $c$ as in \eqref{estobs_c}.
Then, direct computations, along with \eqref{obstime}, imply that
\[
\inf_{ \mc{C} \cap \{ t = 0 \} } f \geq - ( 1 + 2 \kappa )^{-1} \text{,} \qquad \sup_{ \mc{C} \cap \{ t = \pm T \} } f < - ( 1 + 2 \kappa )^{-1} \text{.}
\]
Hence, one can find constants $0 < \delta \ll T$ and $\mu_\kappa > ( 1 + 2 \kappa )^{-1}$ such that
\begin{equation} \label{estobs_f}
\begin{cases}
	f \leq - \mu_\ka & \quad \text{when } t \in ( -T, -T + \delta ) \cup ( T - \delta, T ) \\ 
	f \geq - \mu_\ka & \quad \text{when } t \in ( -\delta, \delta )
\end{cases} \text{.}
\end{equation}

In addition, we define the shorthands
\begin{equation}
\label{estobs_IJ} I_\delta = [ -T + \delta, T - \delta ] \text{,} \qquad J_\delta = ( -T, -T +\delta ) \cup ( T - \de, T ) \text{.}
\end{equation}
We also let $\xi \in C^\infty ( \bar{\mc{C}} )$ be a cutoff function satisfying:
\begin{enumerate}[i)]
\item $\xi$ depends only on $t$.

\item $\xi = 1$ when $t \in I_\delta$.

\item $\xi = 0$ near $t = \pm T$.
\end{enumerate}
We can then apply the Carleman inequality in Theorem \ref{T.Carleman}, with our above choice \eqref{estobs_c} of $c$ and to the function $\xi u$, in order to obtain
\begin{align}
\label{estobs1} &\la \int_\Ga e^{ 2 \lambda f } \xi^2 ( \mc{N}_\ka u )^2 + \int_{ \mc{C} } e^{ 2 \la f } | \Box_\ka ( \xi u ) |^2 \\
\notag &\quad \gtrsim \la \int_{ \mc{C} } e^{ 2 \la f } [ | \partial_t ( \xi u ) |^2 + \xi^2 | \snabla u |^2 + \xi^2 ( D_r u )^2 + \la^2 y^{ -1 + 6 \kappa } \xi^2 u^2 ] \\
\notag &\quad \gtrsim \la \int_{ I_\delta \times B_1 } e^{ 2 \la f } [ ( \partial_t u )^2 + | \snabla u |^2 + ( D_r u )^2 + \la^2 y^{ -1 + 6 \kappa } u^2 ] \text{.}
\end{align}

Moreover, noting that
\begin{align*}
| \Box_\ka ( \xi u ) | &\lesssim | \xi \Box_\ka u | + | \partial_t \xi | \partial_t u | + | \partial_t^2 \xi | | u | \\
&\lesssim | \Box_\ka u | + | \partial_t u | + | u | \text{,}
\end{align*}
and recalling \eqref{hypXV} and \eqref{estobs_f}, we derive that
\begin{align*}
\int_{ \mc{C} } e^{ 2 \la f } | \Box_\ka ( \xi u ) |^2 &\lesssim \int_{ I_\delta \times B_1 } e^{ 2 \la f } | \Box_\ka u |^2 + \int_{ J_\delta \times B_1 } e^{ 2 \lambda f } ( | \Box_\ka u | + | \partial_t u | + | u | ) \\
&\lesssim \int_{ I_\delta \times B_1 } e^{ 2 \lambda f } ( | \partial_t u |^2 + | D_r u |^2 + | \snabla u |^2 ) \\
&\qquad + \int_{ I_\delta \times B_1 } \left( \frac{1}{ y^2 } + \frac{ n - 1 }{ r^2 } \right) ( e^{ \lambda f } u )^2 \\
&\qquad + e^{ - 2 \lambda \mu_\kappa } \int_{ J_\delta \times B_1 } ( | \partial_t u |^2 + | D_r u |^2 + | \snabla u |^2 ) \\
&\qquad + e^{ - 2 \lambda \mu_\kappa } \int_{ J_\delta \times B_1 } \left( \frac{1}{ y^2 } + \frac{ n - 1 }{ r^2 } \right) u^2 \text{,}
\end{align*}
where the implicit constants of the inequalities depend also on $X$ and $V$.
Applying Lemma \ref{L.hardy} and recalling the definition of $f$, the above becomes
\begin{align}
\label{estobs2} \int_{ \mc{C} } e^{ 2 \la f } | \Box_\ka ( \xi u ) |^2 &\lesssim \int_{ I_\delta \times B_1 } [ e^{ 2 \lambda f } ( | \partial_t u |^2 + | D_r u |^2 + | \snabla u |^2 ) + | D_r ( e^{ \lambda f } u ) |^2 ] \\
\notag &\qquad + e^{ - 2 \lambda \mu_\kappa } \int_{ J_\delta \times B_1 } ( | \partial_t u |^2 + | D_r u |^2 + | \snabla u |^2 ) \\
\notag &\lesssim \int_{ I_\delta \times B_1 } e^{ 2 \lambda f } ( | \partial_t u |^2 + | D_r u |^2 + | \snabla u |^2 + \lambda^2 y^{ 4 \kappa } u^2 ) \\
\notag &\qquad + e^{ - 2 \lambda \mu_\kappa } \int_{ J_\delta } E_1 [ u ] ( \tau ) \, d \tau \text{,}
\end{align}

Combining the inequalities \eqref{estobs1} and \eqref{estobs2} and letting $\la$ be sufficiently large (depending also on $X$ and $V$), we then arrive at the bound
\begin{align*}
&\la \int_\Ga e^{ 2 \lambda f } ( \mc{N}_\ka u )^2 + e^{ -2 \la \mu_\ka } \int_{ J_\de } E_1 [u] ( \tau ) \, d \tau \\
\notag &\quad \gtrsim \la \int_{ I_\delta \times B_1 } e^{ 2 \la f } ( | \partial_t u |^2 + | \snabla u |^2 + | D_r u |^2 + \la^2 y^{ 6 \ka - 1 } u^2 )
\end{align*}
Further restricting the domain of the integral in the right-hand side to $( - \delta, \delta ) \times B_1$ and recalling the lower bound in \eqref{estobs_f}, the above becomes
\begin{equation}
\label{obsest3} \la \int_\Ga e^{ 2 \lambda f } ( \mc{N}_\ka u )^2 + e^{ -2 \la \mu_\ka } \int_{ J_\de } E_1 [u] ( \tau ) \, d \tau \gtrsim \la e^{ - 2 \lambda \mu_\ka } \int_{ - \delta }^\delta E_1 [ u ] ( \tau ) \, d \tau \text{.}
\end{equation}

Finally, the energy estimate \eqref{energyineq} implies 
\[
e^{-M T} E_1 [u] (0)\leq E_1[u](t)\leq e^{M T}E_1[u](0) \text{,}
\]
which, when combined with \eqref{obsest3}, yields
\begin{equation}
\label{obsest4} \la \int_\Ga e^{ 2 \lambda f } ( \mc{N}_\ka u )^2 + \delta e^{ -2 \la \mu_\ka } e^{ M T } \cdot E_1 [u] ( 0 ) \gtrsim \la \delta e^{ - 2 \lambda \mu_\ka } e^{ - M T } \cdot E_1 [ u ] ( 0 ) \text{.}
\end{equation}
Taking $\lambda$ in \eqref{obsest4} large enough such that $e^{ 2 M T } \ll \lambda$ results in \eqref{Observability}.

\section*{Acknowledgments}

A.E.\ and B.V.\ are supported by the ERC Starting Grant~633152 and by the
ICMAT--Severo Ochoa grant SEV--2015--0554.
A.S.\ is supported by the EPSRC grant EP/R011982/1.
The authors also thank two anonymous referees for their helpful comments.

\raggedbottom

\end{document}